\documentclass[reqno,12pt,letterpaper]{amsart}
\usepackage{amsmath,amssymb,amsthm,graphicx,mathrsfs,url,bbm,MnSymbol}
\usepackage[usenames,dvipsnames]{color}
\usepackage[colorlinks=true,linkcolor=Red,citecolor=Green]{hyperref}

\def\?[#1]{\textbf{[#1]}\marginpar{\Large{\textbf{??}}}}
\newtheorem{thm}{Theorem}
\newtheorem{prop}{Proposition}

\newtheorem{lem}[prop]{Lemma}
\newtheorem{cor}[prop]{Corollary}
\newtheorem{rem}[prop]{Remark}

\newtheorem{conj}{Conjecture}
\numberwithin{equation}{section}
\numberwithin{prop}{section}

\DeclareMathOperator{\WF}{WF}
\DeclareMathOperator{\Res}{{\rm Res}}

\newcommand{\RR}{{\mathbb R}}
\newcommand{\CC}{{\mathbb C}}

\newcommand{\ZZ}{{\mathbb Z}}
\newcommand{\Ss}{{\mathbb S}}
\newcommand{\Hom}{\mathrm{Hom}}
\newcommand{\id}{\mathrm{Id}}

\begin{document}
\title{Generic twisted Pollicott--Ruelle resonances and zeta function at zero}

\author{Tristan Humbert}
\email{humbertt@imj-prg.fr}
\address{Institut de mathématiques de Paris rive gauche, Sorbonne Université, Campus Pierre et Marie Curie, 4 place Jussieu, 75005, Paris, France.}

\author{Zhongkai Tao}
\email{ztao@ihes.fr}
\address{Institut des Hautes \'Etudes Scientifiques, 35 route de Ghartres, 91440 Bures-sur-Yvette, France}%

\begin{abstract}
For a connected orientable closed surface $(\Sigma,g)$ of genus $G$ with Anosov geodesic flow, we show the existence of an open subset $\mathcal U_g$ of finite-dimensional irreducible representations of the fundamental group of its unit tangent bundle, whose complement has complex codimension at least one and such that for any $\rho \in \mathcal U_g$, the twisted Ruelle zeta function $\zeta_{g,\rho}(s)$ vanishes at $s=0$ to order ${\rm dim}(\rho)(2G-2)$ if $\rho$ factors through $\pi_1(\Sigma)$, and does not vanish otherwise. In the second case, we show that $\zeta_{g,\rho}(0)$ is given by the Reidemeister--Turaev torsion, thus extending Fried's conjecture to a generic set of acyclic (but not necessarily unitary) representations. We also show that the order of vanishing of the untwisted zeta function is constant for an open and dense subset of Anosov metrics in the connected component of a hyperbolic $3$-metric. Our proofs rely on computing the dimensions of the spaces of generalized twisted Pollicott--Ruelle resonant states at zero. 
\end{abstract}

\maketitle
\section{Introduction}
\subsection{Twisted Ruelle zeta function}
Let $(\Sigma,g)$ be a connected orientable closed surface of genus $G\geq 2$ and let $M=S\Sigma:=\{(x,v)\in T\Sigma\mid g(v,v)=1\}$ be its unit tangent bundle. Suppose that the geodesic flow $\varphi_t^g :M\to M$ is \emph{Anosov} (e.g., $g$ is negatively curved). We say that $g$ is an Anosov metric or that $(\Sigma,g)$ is an Anosov surface.

Let $r\in \mathbb N$ and let $\rho \in \mathrm{Hom}(\pi_1(M), \mathrm{GL}_r(\CC))$ be a $r$-dimensional representation of the fundamental group $\pi_1(M)$ of $M$. We define the \emph{twisted Ruelle zeta function} of $(M, g,\rho)$ for any $s\in \mathbb C$ with $\mathrm{Re}(s)\gg 1$, by the infinite product:
\begin{equation}
\label{eq:zeta}
\zeta_{g,\rho}(s)=\prod_{\gamma\in \mathcal{P}} \det(\mathrm{Id}-\rho([\gamma])e^{-s\ell_g(\gamma)}),
\end{equation}
where $\mathcal{P}$ is the set of primitive $g$-geodesics, $[\gamma]$ is the class of $\gamma$ in $\pi_1(M)$ and $\ell_g(\gamma)$ denotes the length of $\gamma$. The zeta function is convergent and holomorphic in a half plane $\{s\in \CC \mid \mathrm{Re}(s)\gg 1\}$ and admits a meromorphic extension to $\mathbb C$, see \S \ref{sec:Resonance}. In this paper, we study the order of vanishing $m(g,\rho)$ of the meromorphic extension of $\zeta_{g,\rho}(s)$ at $s=0.$ 

There is natural map $\pi_*: \pi_1(M)\to \pi_1(\Sigma)$ induced by the projection $\pi:M\to \Sigma$. Any representation of $\pi_1(\Sigma)$ induces a representation of $\pi_1(M)$. In this case, we say that the induced representation \emph{factors through $\pi_1(\Sigma)$}. Note, however, that there are also representations of $\pi_1(M)$ that do \emph{not} factor through $\pi_1(\Sigma)$, see \S \ref{sec:pi1}. We denote by $\mathrm{Hom}_{\mathrm{irr}}(\pi_1(M), \mathrm{GL}_r(\CC))$ the subset of irreducible representations. Note that it is an affine algebraic set by looking at its description using generators and relations. The main result of this paper is the following theorem.  
\begin{thm}
\label{maintheo}
    Let $(\Sigma,g)$ be a connected orientable closed Anosov surface of genus $G\geq 2$. There exists a subset $\mathcal U_g \subset \mathrm{Hom}_{\mathrm{irr}}(\pi_1(M), \mathrm{GL}_r(\CC))$ satisfying the following properties:
\begin{itemize}
    \item the subset $\mathcal U_g$ is open;
    \item its complement $\mathrm{Hom}_{\mathrm{irr}}(\pi_1(M), \mathrm{GL}_r(\CC))\setminus \mathcal U_g$ has complex codimension $\geq 1$;
    \item for any $\rho \in \mathcal U_g$,
    \begin{enumerate}
        \item \label{case1}if $\rho$ factors through $\pi_1(\Sigma)$, one has
    \begin{equation}
    \label{eq:mcase1}
        m(g,\rho)=-\mathrm{dim}(\rho)\chi(\Sigma)=\mathrm{dim}(\rho)(2G-2),
      \end{equation}  where $\chi(\Sigma)$ is the Euler characteristic of $\Sigma$;
    \item \label{case2} if $\rho$ does not factor through $\pi_1(\Sigma)$, one has $m(g,\rho)=0$.
    \end{enumerate}
\end{itemize}
\end{thm}
In Case \ref{case1}, the study of the order of vanishing of $\zeta_{g,\rho}(s)$ at $s=0$ was initiated by Fried \cite[Theorem 1 and Corollary 2]{Fri} for hyperbolic metrics and unitary representations $\rho$ of $\pi_1(\Sigma).$ His result was recently extended by Frahm and Spilioti in \cite[Corollary C]{FS23}. They showed that  for a hyperbolic metric $g$ and any $\rho \in \mathrm{Hom}(\pi_1(\Sigma), \mathrm{GL}_r(\CC))$, one has $m(g,\rho)=-\mathrm{dim}(\rho)\chi(\Sigma)$. When $g$ is hyperbolic, $\zeta_{g,\rho}(s)$ can be expressed as a quotient of shifted {twisted Selberg zeta functions} whose behavior at $s=0$ can be analyzed using a twisted \emph{Selberg trace formula}. The zeros of $\zeta_{g,\rho}(s)$ can be computed from the eigenvalues of the {twisted Laplacian} $\Delta_\rho$ acting on $L^2(\Sigma,\CC)$, see for instance \cite[Theorem 4.2.6]{FS23}.

For an Anosov (not necessarily hyperbolic) metric $g$\footnote{Their work applies to the more general setting of contact Anosov $3$-flows.}  and the trivial representation $\rho_{\mathrm {triv}}$, Dyatlov and Zworski showed in \cite{zazi} that $m(g,\rho_{\mathrm {triv}})=-\chi(\Sigma)$. When $g$ has non-constant curvature, there is no Selberg trace formula and no direct relation to the spectrum of a twisted Laplacian on the base $\Sigma$. They instead use that $m(g,\rho_{\mathrm {triv}})$ can be expressed as an alternate sum of dimensions of generalized eigenspaces of the Anosov geodesic vector field $X$ acting on specially designed \emph{anisotropic spaces}. The formula for $m(g,\rho_{\mathrm {triv}})$ then follows from an explicit computation of these dimensions. 

The order of vanishing of $\zeta_{g,\rho}(s)$ for general Anosov $3$-flows was studied by Cekić and Paternain in \cite{CP20,CP25}. In particular, \eqref{eq:mcase1} was obtained in \cite[Corollary 1.9]{CP20} under the condition that $\rho$ is unitary. In higher dimensions, the order of vanishing of the (non-twisted) zeta function near $3$-hyperbolic metrics was studied by Cekić, Delarue, Dyatlov and Paternain \cite{zeta3}. They showed that the order of vanishing of the zeta function at $s=0$ is \emph{not} a topological invariant \cite[Theorem 1]{zeta3} but conjectured that it should still be constant on a set of generic metrics \cite[Conjecture 1]{zeta3}. 

Case \ref{case2} is related to the \emph{Fried conjecture.} Fried's conjecture \cite{Fri86,Fri95} states that for an acyclic unitary representation $\rho$ of $\pi_1(M)$, the twisted Ruelle zeta function $\zeta_{g,\rho}(s)$ is well-defined and nonvanishing at $0$, and
\begin{equation}
\label{eq:Fried}
    |\zeta_{g,\rho}(0)|^{(-1)^{\frac{\dim M-1}{2}}}=\tau_{\rho}(M),
\end{equation}
where $\tau_{\rho}(M)$ is the Reidemeister torsion. It is known for locally symmetric spaces \cite{Fri86,Fri95,MS91,She17}. When $\dim (M)= 3$, it is known for unitary representations under some mild assumptions \cite{SM96,DGRS}, see Shen \cite{She21} for a recent survey. It is natural to ask if an analog of \eqref{eq:Fried} holds for non-unitary acyclic representations on $M=S\Sigma$ with respect to the geodesic flow, see for instance \cite{Mue20,BFS23,CD24,BS25}. For Fried's conjecture for pseudo-Anosov flows, see \cite{JZ24}.

Theorem \ref{maintheo} computes the order of vanishing of the zeta function at $s=0$ for an Anosov metric $g$, and for a generic representation $\rho\in \mathrm{Hom}_{\mathrm{irr}}(\pi_1(M), \mathrm{GL}_r(\CC))$. The main novelty from the previously cited works is that we consider \emph{non-unitary} twists $\rho$ of \emph{non-hyperbolic} metrics $g$. 
\subsection{Twisted Pollicott--Ruelle resonances.} 
\label{sec:introRuelle}Theorem \ref{maintheo} follows from a more general statement about \emph{twisted Pollicott--Ruelle resonances}, see Theorem \ref{theo2} below.

 Let $X=\tfrac d{dt}|_{t=0}\varphi_t^g$ denote the generator of the geodesic flow $(\varphi_t^g)_{t\in \mathbb R}.$
We write $d^{\nabla_\rho}$ for the differential on the flat vector bundle $\mathcal{E}_{\rho}$ induced by $\rho\in \mathrm{Hom}( \pi_1(M), \mathrm{GL}_r(\CC))$, see \S \ref{sec:repr} for the precise definitions. The vector field $X$ lifts to the bundle $\mathcal{E}_{\rho}$ as the Lie derivative  $\mathcal{L}_{X^\rho} =\iota_X d^{\nabla_\rho}+d^{\nabla_\rho} \iota_X$, where $\iota_X$ denotes the contraction by~$X$.

 For $k=0,1,2$, denote by $\Omega^k_0$ the space of smooth $k$-differential forms on $M$ that are in the kernel of $\iota_X$. Note that the action of $\mathcal L_{X^\rho}$ extends to $\Omega_0^k\otimes \mathcal E_\rho.$ Since $(\varphi_t^g)_{t\in \mathbb R}$ is Anosov, one can associate to $\mathcal L_{X^\rho}$ a discrete spectrum $\mathrm{Res}^k(X,\rho)\subset \CC$, the \emph{twisted Pollicott--Ruelle resonances}, by making $\mathcal L_{X^\rho}$ act on \emph{anisotropic spaces}, see \S \ref{sec:Resonance} for a precise definition. For a twisted Pollicott--Ruelle resonance $\lambda \in \CC$, the corresponding (generalized) eigenvectors are called (generalized) \emph{resonant states} at $\lambda$. Let $\mathrm{Res}^{k,1}_0(\rho,0)$ (resp.  $\mathrm{Res}^{k,\infty}_0(\rho,0)$) denote the space of \emph{resonant states at $0$} (resp. \emph{generalized resonant states at $0$}) for the action of $\mathcal L_{X^\rho}$ on $\Omega_0^k\otimes \mathcal E_\rho$. The order of vanishing at $s=0$ of $\zeta_{g,\rho}(s)$ is given by 
\begin{equation}
\label{eq:orderofvanish}
    m(g,\rho)=\mathrm{dim} \big(\Res_{0}^{1,\infty}(\rho,0)\big) - \mathrm{dim} \big(\Res_{0}^{0,\infty}(\rho,0)\big)-\mathrm{dim}\big(\Res_{0}^{2,\infty}(\rho,0)\big),
\end{equation}
see \S \ref{sec:Resonance} for a proof. We show the following result.

\begin{thm}
\label{theo2}
Let $(\Sigma,g)$ be a connected orientable closed Anosov surface of genus $G\geq 2$. There exists a subset $\mathcal U_g \subset \mathrm{Hom}_{\mathrm{irr}}(\pi_1(M), \mathrm{GL}_r(\CC))$ satisfying the following properties:
\begin{itemize}
    \item the subset $\mathcal U_g$ is open;
    \item its complement $\mathrm{Hom}_{\mathrm{irr}}(\pi_1(M), \mathrm{GL}_r(\CC))\setminus \mathcal U_g$ has complex codimension $\geq 1$;
    \item for any $\rho \in \mathcal U_g$, 
    \begin{enumerate}
        \item if $\rho$ factors through $\pi_1(\Sigma)$, one has
    \begin{equation}
    \label{eq:dim}
    \begin{split}
        &\mathrm{dim}(\Res_0^{k,\infty}(\rho,0))=0, \quad k=0,2, 
        \\ &\mathrm{dim}(\Res_0^{1,\infty}(\rho,0))=\mathrm{dim}(\Res_0^{1,1}(\rho,0))=-\mathrm{dim}(\rho)\chi(\Sigma). 
        \end{split}
          \end{equation}
        \item if $\rho$ does not factor through $\pi_1(\Sigma)$, one has
        \begin{equation}\label{eq:dim2}
          \Res_0^{k,\infty}(\rho,0)=0, \quad k=0,1,2. 
        \end{equation}
    \end{enumerate}
\end{itemize}
\end{thm}
It is clear from \eqref{eq:orderofvanish} that Theorem \ref{maintheo} follows from Theorem \ref{theo2}. Theorem \ref{theo2} shows that for a generic $\rho\in \mathrm{Hom}(\pi_1(M),\mathrm{GL}_r(\CC))$, the space of resonant states at $0$ for $k=0,2$ is trivial and $\mathcal L_{X^\rho}$ has no Jordan block for $k=1$, with a corresponding space of resonant states of dimension $-\mathrm{dim}(\rho)\chi(\Sigma)$ in Case \ref{case1} and $0$ in Case \ref{case2}. 

For an acyclic representation $\rho$, one can define the \emph{Reidemeister--Turaev torsion}: $$\tau_{\mathfrak{e}_{\mathrm{geod}},\mathfrak{o}}(\rho)\in \CC\setminus \{0\}.$$ Here, $\mathfrak{e}_{\mathrm{geod}}$ is the Euler structure induced by the geodesic vector field, and $\mathfrak{o}$ is a homological orientation, see \cite[\S 10]{CD24} and \cite[\S 2]{BFS23} for the detailed definitions. In particular, when $\rho$ is unitary, one has $|\tau_{\mathfrak{e}_{\mathrm{geod}},\mathfrak{o}}(\rho)|=\tau_\rho(M)$ where $\tau_\rho(M)$ is the Reidemeister torsion. In \cite[Theorem A]{BFS23}, the authors showed that for $g_0$ hyperbolic and $\rho$ an irreducible representation which does not factor through $\pi_1(\Sigma)$\footnote{Such a representation has to be acyclic, see Lemma \ref{lemm:coho}.}, one has
\begin{equation}
    \zeta_{g_0,\rho}(0)^{-1}=\pm \tau_{\mathfrak{e}_{\mathrm{geod}},\mathfrak{o}}(\rho)=\pm \det(\id-\rho(c))^{2G-2},
\end{equation}
where $c\in \pi_1(M)$ is defined in \eqref{eq:pi1pre}.
This can be seen as a generalization of \eqref{eq:Fried} in the non-unitary case. In this paper, we extend their result for a non-hyperbolic metric $g$ and a set of generic representations.
 
\begin{cor}
\label{cor}
    Let $(\Sigma,g)$ be a connected orientable closed negatively curved surface of genus $G\geq 2$. There exists a subset $\mathcal U_g \subset \mathrm{Hom}_{\mathrm{irr}}(\pi_1(M), \mathrm{GL}_r(\CC))$ satisfying the following properties:
\begin{itemize}
    \item the subset $\mathcal U_g$ is open;
    \item its complement $\mathrm{Hom}_{\mathrm{irr}}(\pi_1(M), \mathrm{GL}_r(\CC))\setminus \mathcal U_g$ has complex codimension $\geq 1$;
    \item for any $\rho \in \mathcal U_g$, 
    which does not factor through $\pi_1(\Sigma)$, one has
       \begin{equation}\label{eq:Reide-Tur}
           \zeta_{g,\rho}(0)^{-1}=\pm \tau_{\mathfrak{e}_{\mathrm{geod}},\mathfrak{o}}(\rho)=\pm \det(\id-\rho(c))^{2G-2}.
       \end{equation}
 \end{itemize}
\end{cor}
\subsubsection{Dimensions of the space of generalized resonant states.} We note that \eqref{eq:dim} is not satisfied for every $\rho\in \pi_1(M)$ in Case \ref{case1}. Indeed, for $\rho=\rho_{\mathrm{triv}}$, Dyatlov and Zworski  showed in \cite[Proposition 3.1]{zazi} that
\begin{align*}&\mathrm{dim}(\Res_0^{k,1}(\rho_{\mathrm{triv}},0))=1, \quad k=0,2, 
\\ &\mathrm{dim}(\Res_0^{1,\infty}(\rho_{\mathrm{triv}},0))=\mathrm{dim}(\Res_0^{1,1}(\rho_{\mathrm{triv}},0))=b_1(M), 
\end{align*}
where $b_1(M)$ is the first Betti number of $M$. Moreover, using the work of Naud and Spilioti \cite{NS22}, we obtain the following result.
\begin{thm}
\label{theo:3}
    Let $(\Sigma,g)$ be a connected orientable closed hyperbolic surface of genus $G$. Let $\rho={\rm Ad}$ be the adjoint representation of $\mathrm{SL}(2,\mathbb R)$. Then 
    \begin{equation}
        \label{eq:orderofvan}
        \begin{split}
&\mathrm{dim}(\Res_{0}^{k,\infty}({\rm Ad},0))=2G+1,\quad k=0,2;
\\&\mathrm{dim}(\Res_{0}^{1,\infty}({\rm Ad},0))=10G-4.
\end{split}
    \end{equation}
\end{thm}
Theorem \ref{theo:3} shows that the spaces of generalized resonant states need not be trivial for $k=0,2$ even when $\rho$ is a nontrivial and irreducible representation. Moreover, the dimensions can be as large as we want when $G\to+\infty.$ 
\subsubsection{Presence of Jordan blocks.}
We prove in Proposition \ref{prop:noJD} that if there is no Jordan block at zero for $k=0,1,2$, then  $m(g,\rho)=\dim(\rho)(2G-2°)$ in Case \ref{case1} and $m(g,\rho)=0$ in Case \ref{case2}.

However, we obtain that in Case \ref{case2}, there exist pairs $(g,\rho)$ for which there is a Jordan block at zero, which shows that the general picture is more complicated than the generic one depicted in Theorem \ref{theo2}.
For a metric $g$ on $\Sigma$, we denote by $\Delta_g$ its (positive) untwisted Laplace--Beltrami operator acting on $L^2(\Sigma,\CC)$.
\begin{thm}
\label{theo:4}
    Let $(\Sigma,g)$ be a connected orientable closed hyperbolic surface such that $\tfrac 14 \in \mathrm{Spec}(\Delta_g)$. Then there exists an irreducible representation $\tau$ of $\pi_1(M)$ which does not factor through $\pi_1(\Sigma)$, for which $m(g,\tau)=0$ and such that $\mathcal L_{X^{\tau}}$ has a non-trivial Jordan block at zero for $k=0,1,2$.

\end{thm}
The representation $\tau$ is constructed explicitly in \S \ref{sec:JD}. Using the quantum-classical correspondence of Guillarmou, Hilgert and Weich \cite{GHW}, we give in Proposition \ref{prop:JDtau} an explicit description of the Jordan block structure of $(g,\tau)$ at zero in terms of $\mathrm{ker}(\Delta_g-\tfrac 14)$. Finally, we show in Proposition \ref{prop:1/4} that for any $G\geq 2$, there exists a hyperbolic surface of genus $G$ for which $\tfrac 14 \in \mathrm{Spec}(\Delta_g)$. 

Theorem \ref{theo:4} is, to the best of our knowledge, the first example of an Anosov flow and acyclic representation with a non-trivial Jordan block at zero. For unitary representations in our setting, the resonant spaces were shown to be trivial in \cite{DGRS}. Theorem~\ref{theo:4} shows that this is not always the case for non-unitary twists.

\subsection{Generic semisimplicity}
Next, we consider $(\Sigma,g)$, an orientable connected closed manifold of dimension $n+1\geq 2$ with an Anosov metric, and a representation $\rho:\pi_1(M) \to \mathrm{GL}_r(\CC)$ for $M=S\Sigma$. As explained before, the order of vanishing of $\zeta_{g,\rho}$ in \eqref{eq:zeta} at $s=0$ was shown to \emph{not} be a topological invariant in dimension $n+1=3$ for the trivial representation $\rho=\rho_{\mathrm{triv}}$ in \cite{zeta3}. In this paper, we show the following dichotomy on the order of vanishing of $\zeta_{g,\rho}$ at $s=0.$

\begin{thm}\label{thm:sim}
Let $\mathcal{V}$ be a connected open set of Anosov metrics on $\Sigma$, and $\rho:\pi_1(M) \to \mathrm{GL}_r(\CC)$. 
\begin{enumerate}
    \item Either the following holds for an open and dense set of set of metrics $g$ in $\mathcal{V}$
    \begin{equation}\label{eq:cond-simp}
    \Res_0^{k,\infty}(g,\rho,0)=\Res_0^{k,1}(g,\rho,0),\quad
    d^{\nabla_{\rho}}(\Res_0^{k,1}(g,\rho,0))=0,\quad k=0,1,2,\cdots, 2n,
\end{equation}
and
\begin{equation}\label{eq:cond-dim}
\begin{split}
    \dim \Res_0^{k,1}(g,\rho, 0) &= \sum_{j=0}^{\lfloor k/2\rfloor} \dim H^{k-2j}(M,\rho), \\
    \dim \Res_0^{k,1}(g,\rho, 0)&= \dim \Res_0^{2n-k,1}(g,\rho, 0), \quad 0\leq k\leq n.
\end{split}
\end{equation}
 In particular, one has
\begin{equation}
    \label{eq:ovanishgen}
    m(g,\rho)=\sum_{k=0}^n(-1)^{k+n}(n+1-k)\,\mathrm{dim} H^k(M,\rho).
\end{equation}
\item Or \eqref{eq:cond-simp} does not hold for any $g\in \mathcal{V}$.
\end{enumerate}
\end{thm}
By a similar argument as in \cite{zeta3}, \eqref{eq:cond-simp} implies \eqref{eq:cond-dim} (see also Lemma~\ref{lem:ind}). In \cite[Conjecture 1]{zeta3}, the authors conjectured that the second case never happen for the trivial representation, that is, \eqref{eq:cond-simp} holds for an open and dense set of Anosov metrics on $\Sigma.$ In \cite[Theorem 1]{zeta3}, they showed that the conjecture is true for a generic conformal perturbation of a hyperbolic $3$-metric. As a direct consequence of Theorem \ref{thm:sim}, we deduce the following result.

\begin{cor}\label{cor:hyp}
Suppose $(\Sigma,g_0)$ is a compact connected oriented hyperbolic $3$-manifold and $\rho=\rho_{\mathrm{triv}}$ be the trivial representation. Let $\mathcal{V}$ be the connected component of Anosov metrics containing $g_0$. Then for an open and dense set of Anosov metrics $g$ in $\mathcal{V}$, we have \eqref{eq:cond-simp}, \eqref{eq:cond-dim} and \eqref{eq:ovanishgen}.
\end{cor}
Corollary \ref{cor:hyp} extends \cite[Theorem 1]{zeta3} to the whole connected component of $g_0$. We do not require $g$ to be a perturbation of $g_0$, and neither do we require $g$ to be conformal to $g_0.$ To the best of our knowledge, it is not known whether the space of Anosov (or negatively curved) metrics on a hyperbolic 3-manifold is connected or not, see Conjecture \ref{conj2}.

Finally, we give another example of representations that do not satisfy \eqref{eq:cond-dim} (which in particular applies to the case of $\Sigma$ a connected orientable negatively curved surface).
\begin{prop}\label{prop:per-triv}
    Let $r\geq 1$ and $\Sigma$ be a closed orientable manifold with an Anosov metric $g$. Suppose $M=S\Sigma$ satisfies $ \dim H_1(M,\CC) \geq 2$. Then there exists a representation $\pi_1(M)\to \mathrm{GL}_r(\CC)$ such that
    \begin{equation}
        \dim \Res_0^{0,\infty}(g,\rho,0) \neq \dim H^0(M,\rho).
    \end{equation}
\end{prop}

\subsection{Further questions}
Let $\Sigma$ be a genus $G\geq 2$ closed surface. We are not aware of an example of a pair $(g,\rho)$ for which $m(g,\rho)$ is not given by $\dim(\rho)(2G-2)$ in Case \ref{case1} and $0$ in Case \ref{case2}. We make the following conjecture.
\begin{conj}
\label{conj}Let $\Sigma$ be a genus $G\geq 2$ closed surface.
\begin{enumerate}
    \item  There exists a negatively curved metric $g$ on $\Sigma$ and a finite dimensional representation $\rho \in\mathrm{Hom}(\pi_1(\Sigma), \mathrm{GL}_r(\CC)) $ for which $m(g,\rho)\neq (2G-2)r$.
    \item 
    If $\rho$ is acyclic, i.e., $H^*(M,\rho)=0$, then we always have $m(g,\rho)=0$.
\end{enumerate}
\end{conj}
It is interesting to ask what happens for a metric perturbation of $(g,\tau)$ in Theorem~\ref{theo:4}.

Motivated by Corollary~\ref{cor:hyp}, we make the following conjecture.
\begin{conj}
\label{conj2}
    Let $(\Sigma,g_0)$ be a negatively curved oriented closed $3$-manifold. Then the space of Anosov metrics on $\Sigma$ is connected.
\end{conj}
We note that the connectedness of the space of negatively curved metrics on hyperbolic $3$-manifolds would follow from the conjecture that the \emph{cross curvature flow}\footnote{In \cite{CH04}, the authors show that the cross curvature flow preserves negative curvature when defined.} defined from any negatively curved metric exists for all time and converges to a hyperbolic metric, see \cite{CH04} for partial results on short time existence.

\subsection{Structure of the paper}
In \S \ref{sec:repr}, we recall the construction of the flat bundle associated to a representation $\rho$. Next, we compute in Lemma \ref{lemm:coho} the dimensions of the first twisted cohomology groups of $(M,\rho)$ in Cases \ref{case1} and \ref{case2}. We define in Proposition \ref{prop:connected2} a  Zariski open subset $\mathcal{V}\subset \Hom_{\mathrm{irr}}(\pi_1(M), \mathrm{GL}_r(\CC))$ which decomposes as a union of a finite number of path-connected components which all contain an irreducible representation. In \S \ref{sec:Resonance}, we recall the definition of twisted Pollicott--Ruelle resonances and prove \eqref{eq:orderofvanish}.

In \S \ref{sec:pert}, we show that flat bundles $\mathcal E_\rho$ for different $\rho$ can be identified. Moreover, the identification can be chosen to depend analytically on $\rho$, see Lemma \ref{lem:ident}. This allows us to use perturbation theory for Pollicott--Ruelle resonances and show that  the spectral projectors at zero depend analytically in $\rho$, see Proposition \ref{prop:pert}.

In \S \ref{sec:ofv}, we adapt the argument of \cite{zazi} to compute $m(g,\rho)$ for representations with no Jordan block at zero, see Proposition \ref{prop:noJD}. The end of the section is dedicated to the showing that unitary representations have no Jordan block at zero, see Proposition \ref{prop:DZ}.

We prove Theorem \ref{theo2} in \S \ref{sec:proofMaintheo}. The strategy is the following:
\begin{itemize}
    \item For each $\rho \in \mathcal V$, there is a path $(\rho(t))_{t\in[0,1]}\subset \mathcal V$ such that $\rho(1)=\rho$ and $\rho(0)$ is unitary.
    \item We use the perturbation theory developed in \S \ref{sec:pert} to show that for  each $s\in [0,1]$, $\rho(s)$ has a neighborhood $\mathcal U_{\rho(s)}$ in which the conclusion of Theorem \ref{theo2} does \emph{not} hold for a Zariski closed subset.
    \item Using the fact that $\rho(0)$ is unitary and Proposition \ref{prop:DZ}, we show that this Zariski closed subset is proper and thus of complex codimension $\geq 1$.
\end{itemize}
We show Corollary \ref{cor} at the end of the section, using \cite{CD24} and \cite{BFS23}.

Theorems \ref{theo:3} and \ref{theo:4} are proved in \S \ref{sec:last}. Theorem \ref{theo:3} follows from results of \cite{NS22} on the twisted Selberg zeta function. To obtain Theorem \ref{theo:4}, we use the quantum--classical correspondence obtained in \cite{GHW}.

Finally, Theorem \ref{thm:sim} and Proposition \ref{prop:per-triv} are shown in \S \ref{sec:GenSim} following a similar strategy as in the proof of Theorem \ref{theo2}.

\subsection*{Acknowledgements} We would like to thank Mihajlo Cekić, Yann Chaubet, Semyon Dyatlov, Yulin Gong, Colin Guillarmou, Thibault Lefeuvre, Laura Monk, Gabriel Rivière, Yuhao Xue, and Maciej Zworski for helpful discussions related to the project.

T.H. was supported by the European
Research Council (ERC) under the European Union’s Horizon 2020 research and innovation
programme (Grant agreement no. 101162990 — ADG).
\section{Preliminaries}
\subsection{Representations of the fundamental group.} 
\subsubsection{Flat bundles}
\label{sec:repr} Let $(\Sigma,g)$ be a closed Riemannian manifold of dimension $n+1\geq 2$. Let $M=\{(x,v)\in T\Sigma \mid \|v\|_g=1\}$ be its unit tangent bundle. Let $\pi_1(M)$ be the fundamental group of $M$. Recall that $\pi_1(M)$ acts (on the left) on the universal cover $\widetilde{M}$ by deck transformations. Note that the set of deck transformations induced by $\pi_1(M)$ is isomorphic to the opposite group $\pi_1(M)^{\mathrm{op}}$ of $\pi_1(M)$ and we will identify these two groups by a small abuse of notation. Let $\rho \in \mathrm{Hom}(\pi_1(M), \mathrm{GL}_r(\CC))$ be a finite dimensional representation of $\pi_1(M).$
The \emph{flat vector bundle} $\mathcal E_\rho$ is defined as
\begin{equation}
    \label{eq:flateq}
    \mathcal E_\rho:=\widetilde{M}\times \mathbb C^r /\sim ,\quad (x,v)\sim (\gamma(x),\rho(\gamma)v), \ \forall \gamma \in \pi_1(M).
\end{equation}
 The bundle $\mathcal E_\rho$ is equipped with a \emph{flat connection} $\nabla_\rho.$ We recall here its definition. Consider the trivial connection $\nabla^{\rm{triv}}$ on $\widetilde{M}\times \mathbb C^r$:
$$\forall s\in C^{\infty}(\widetilde{M};\widetilde{M}\times \mathbb C^r)=C^{\infty}(\widetilde{M},\mathbb C^r),\quad \nabla^{\rm{triv}}s:=ds\in C^{\infty}(\widetilde{M};T^*\widetilde{M}\otimes \CC^r). $$
We note that the trivial connection descends to $\mathcal E_\rho$. Indeed, recall that a section $s$ of $\mathcal E_\rho$ identifies to a function $C^{\infty}(\widetilde{M},\mathbb C^r)$ which is $\pi_1(M)$-equivariant, i.e, $s(\gamma \cdot x)=\rho(\gamma)s(x)$ for any $x\in \widetilde{M}$ and $\gamma \in \pi_1(M)$. In this identification, we check that
$$\forall x\in\widetilde{M},\ \forall \gamma \in \pi_1(M),\quad d(\rho(\gamma)s(\gamma^{-1}\cdot x))=\rho(\gamma)ds(\gamma^{-1}\cdot x). $$
The induced connection is denoted $\nabla_\rho$ and is clearly flat. We conversely check that the holonomy of $\mathcal E_\rho$ is given by $\rho$.\footnote{From \eqref{eq:flateq}, one sees that the holonomy along a closed geodesic $\gamma$ is given by $v\mapsto \rho^{-1}([\gamma]^{\mathrm{op}})$, where $[\gamma]^{\mathrm{op}}\in \pi_1(M)^{\rm{op}}$ is the class defined by $\gamma$ in the group of deck transformations. Since $[\gamma]^{\mathrm{op}}\cong [\gamma]^{-1}$, where $[\gamma]$ is the class of $\gamma$ in $\pi_1(M)$, we get that the holonomy, seen as a representation of $\pi_1(M)$, is given by $\rho.$} We can define a \emph{flat differential} $d^{\nabla_\rho}$ from $\nabla_{\rho}$. For $k=0,1,\ldots,2n+1$, let $\Omega^k=C^{\infty}(M;\Lambda^kT^*M)$ denote the space of smooth $k$-forms on $M$. The flat differential acts on smooth $k$-forms with values in $\mathcal E_\rho$.
$$d^{\nabla_\rho}:C^{\infty}(M; \Lambda^kT^*M\otimes \mathcal E_{\rho})\to C^{\infty}(M; \Lambda^{k+1}T^*M\otimes \mathcal E_{\rho}). $$
It is uniquely determined by the requirements that it coincides with $\nabla_\rho$ on $0$-forms and that it satisfies the \emph{Leibniz rule:}
\begin{equation}
\label{eq:Leib}
    \forall \omega,\eta \in C^{\infty}(M, \Lambda^\bullet T^*M\otimes \mathcal E_{\rho}),\quad d^{\nabla_\rho}(\omega \wedge \eta)=d^{\nabla_\rho}\omega \wedge \eta+(-1)^{\mathrm{deg}(\omega)}\omega \wedge d^{\nabla_\rho}\eta. 
\end{equation} 
\subsubsection{Twisted cohomology groups}
\label{sec:pi1}
Since the connection is flat, one has $d^{\nabla_\rho}\circ d^{\nabla_\rho}=0$. In particular, we can consider the \emph{twisted de Rham cohomogy} of $M$. For $k=0,1,\ldots,2n+1$, denote by $H^k(M,\rho):=H^k(M,\mathcal E_\rho)$ the twisted cohomology groups of $M$ by $\rho$.

Suppose now that $\Sigma$ is an orientable surface of genus $G\geq 2.$
The natural projection $\pi:M\to \Sigma$ induces a map $\pi_*:\pi_1(M)\to \pi_1(\Sigma).$ Since $M\to \Sigma$ is a $\mathbb S^1$-bundle, we have the following short exact sequence:
\[
1 \longrightarrow \pi_1(\mathbb S^1) \cong \mathbb{Z}
\longrightarrow \pi_1(M)
\xrightarrow{\;\pi_*\;}
\pi_1(\Sigma)
\longrightarrow 1,
\]
where the map $\pi_*$ is surjective. Recall that the groups $\pi_1(\Sigma)$ and $\pi_1(M)$ have the following presentations:
\begin{equation}\label{eq:pi1pre}
\begin{split}
    &\pi_1(\Sigma)=\langle a_1,b_1,\cdots, a_G, b_G\, | \, [a_1,b_1]\cdots [a_G,b_G]=1\rangle,\\
    &\pi_1(M)=\langle a_1,b_1,\cdots, a_G, b_G, c\, |\, [a_i,c]=[b_i,c]=1,\, [a_1,b_1]\cdots [a_G,b_G]=c^{2G-2}\rangle,
\end{split}
\end{equation}
where $c$ is a generator of $\pi_1(\mathbb S^1)\cong \mathbb{Z}$. 

A representation $\rho:\pi_1(M)\to \mathrm{GL}_r(\CC)$ is said to \emph{factor through} $\pi_1(\Sigma)$ if there exists a representation $\bar{\rho}:\pi_1(\Sigma)\to \mathrm{GL}_r(\CC) $ such that $\rho=\bar{\rho}\circ\pi_*.$ Note that $\rho$ factors over $\pi_1(\Sigma)$ if and only if $\rho(c)=\mathrm{Id}.$

For a representation $\rho:\pi_1(M)\to \mathrm{GL}_r(\CC) $, let $\rho^{\pi_1(M)}:=\{v\in \CC^r \mid \forall \gamma \in \pi_1(M),\ \rho(\gamma)v=v\}$ be the invariant subspace. With these notations, we compute in the next lemma the dimensions of the first twisted cohomology groups of $(M,\rho).$
\begin{lem}
\label{lemm:coho} 
    Let $\Sigma$ be a closed surface of genus $G\geq 2$ and let $\rho \in \mathrm{Hom}(\pi_1(M), \mathrm{GL}_r(\CC))$ be a finite dimensional representation of $\pi_1(M)$. 
    \begin{enumerate}
        \item If $\rho$ factors through $\pi_1(\Sigma)$, and $\dim(\rho^{\pi_1(M)})=r_1$. Then
        \begin{equation*}
            \dim(H^0(M,\rho))=r_1,\quad  \dim (H^1(M,\rho))=-r \chi(\Sigma)+2r_1=r(2G-2)+2r_1.
        \end{equation*}
        \item If $\rho$ is irreducible and does not factor through $\pi_1(\Sigma)$, then $\rho$ is acyclic, i.e.,
        \begin{equation*}
            H^i(M,\rho)=0,\quad i=0,1,2,3.
        \end{equation*}
    \end{enumerate}
\end{lem}
\begin{proof}
    Suppose that we are in case $(1)$, i.e., there exists a representation $\bar{\rho}:\pi_1(\Sigma)\to \mathrm{GL}_r(\CC) $ such that $\rho=\bar{\rho}\circ\pi_*.$ One can define a flat bundle $E_{\bar{\rho}}\to \Sigma$ using the same construction recalled above. Moreover, we check that $\mathcal E_\rho=\pi^*E_{\bar \rho}$, where $\pi:M\to \Sigma$ is the projection. Since $M\to \Sigma$ is a circle bundle, applying the Gysin exact sequence gives for any $k$:
$$\ldots \to H^k(\Sigma, E_{\bar \rho})\xrightarrow{\pi^*} H^k(M,\mathcal E_\rho)\xrightarrow{\pi_*}H^{k-1}(\Sigma, E_{\bar \rho}) \xrightarrow{\cup e}H^{k+1}(\Sigma, E_{\bar \rho})\to \ldots,  $$
where $\pi_*$ is the integration in the circle fiber and $\cup e$ denotes the cup product with the Euler class $e=\chi(\Sigma)\in H^2(\Sigma,\mathbb Z)\cong \ZZ.$ For $k=0$, we obtain
$$0\to H^0(\Sigma, E_{\bar \rho})\xrightarrow{\pi^*} H^0(M,\mathcal E_\rho)\xrightarrow{\pi_*}0, $$
which gives
$$
  H^0(M,\rho)=H^0(M,\mathcal E_{ \rho})  \cong H^0(\Sigma, E_{\bar \rho})=\{v\in \CC^r \mid \forall \gamma \in \pi_1(\Sigma),\ \bar{\rho}(\gamma)v=v\}.
$$
This then implies
$$\dim(H^0(M,\rho))=\dim \{v\in \CC^r \mid \forall \gamma \in \pi_1(\Sigma),\ \bar{\rho}(\gamma)v=v\} =r_1. $$
 Next, we compute the groups in degree $k=1$:
$$0\to H^1(\Sigma, E_{\bar \rho})\xrightarrow{\pi^*} H^1(M,\mathcal E_\rho)\xrightarrow{\pi_*}H^0(\Sigma, E_{\bar \rho})\xrightarrow{\cup e}H^{2}(\Sigma, E_{\bar \rho}). $$
In particular, applying the rank-nullity theorem to $\pi_*$ yields 
$$\dim(H^1(M,\rho))= \dim H^1(\Sigma, E_{\bar \rho})+\dim \big(\mathrm{ker}(\cup e:H^0(\Sigma, E_{\bar \rho})\to H^{2}(\Sigma, E_{\bar \rho}))\big).$$
Note however that the cup product with the Euler class $e\neq 0$ is injective on $H^0(\Sigma, E_{\bar \rho})$ so $\dim(H^1(M,\rho))= \dim H^1(\Sigma, E_{\bar \rho})$. We can compute this last dimension from Poincaré duality which gives $H^0(\Sigma, E_{\bar\rho})\cong H^2(\Sigma, E_{\bar \rho}) \cong \CC^{r_1}$ and the definition of the twisted Euler characteristic of $(\Sigma,\rho)$:
$$\chi(\Sigma,\rho)=\sum_{i=0}^2(-1)^i\mathrm{dim}\big(H^i(\Sigma, E_{\bar \rho})\big)\ \Rightarrow \  \dim H^1(\Sigma, E_{\bar \rho})=2r_1-\chi(\Sigma,\rho).$$
 Since the bundle is flat, $\chi(\Sigma,\rho)$ is equal to  $\mathrm{dim}(\rho)\chi(\Sigma)=r(2-2G)$ which concludes the proof of the lemma in the first case.

 Suppose now that $\rho$ is irreducible and does not factor through $\pi_1(\Sigma)$. We will show that it is acyclic. First, by Schur's lemma, we know that $\rho(c)=\zeta \id$ for some $\zeta \in \CC$. Taking the determinant in \eqref{eq:pi1pre} implies that $\zeta$ is a root of unity. Moreover, since $\rho$ does not factor through $\pi_1(\Sigma)$, we have $\zeta\neq 1$. This implies the differential in the fiber is null homotopic. Indeed, let $\iota: \mathbb S^1 \hookrightarrow M$ be the inclusion of a circle fiber. The twisted differential along the fiber is
 $$ d_{\mathbb S^1}^{\nabla}:\Omega^{0}(\mathbb S^1, \iota^*\mathcal E_\rho)\to \Omega^{1}(\mathbb S^1, \iota^*\mathcal E_\rho).$$
 One can compute the twisted cohomology groups of the fiber from the monodromy:
 $$H^0(\mathbb S^1,\iota^*\mathcal E_\rho)=\mathrm{ker}(\rho(c)-\mathrm{Id})=0,\quad  H^1(\mathbb S^1,\iota^*\mathcal E_\rho)=\mathrm{coker}(\rho(c)-\mathrm{Id})=0.$$
This shows that $d_{\mathbb S^1}^{\nabla}:\Omega^{0}(\mathbb S^1, \iota^*\mathcal E_\rho)\to \Omega^{1}(\mathbb S^1, \iota^*\mathcal E_\rho)$ is an isomorphism.

 We can then define a chain homotopy fiberwisely to show the de Rham complex of $(M,d^{\nabla_\rho})$ is also exact. More precisely, suppose $d^{\nabla_\rho} f=0$ and $f\in \Omega^k\otimes \mathcal{E}_{\rho}$ for some $k\geq 1$. The goal is to find $u\in \Omega^{k-1}\otimes \mathcal{E}_{\rho}$ with $d^{\nabla_\rho}u =f$. Let $V$ be a nonvanishing vector field tangent to the fiber direction. The vanishing of the fiberwise cohomology implies that
 \begin{equation*}
\mathcal{L}_V^{\nabla}:=d^{\nabla_\rho}\iota_V+\iota_V d^{\nabla_\rho}
 \end{equation*}
 is invertible on $\Omega^k\otimes \mathcal{E}_{\rho}$. In order to check this claim, we note that $\mathcal{L}_V^{\nabla}$ acts on each fiber and the action is given by $d^{\nabla}_{\Ss^1}\iota_V+\iota_V d_{\Ss^1}^{\nabla}$. On zero forms, it is given by $\iota_V d_{\Ss^1}^{\nabla}$ where each map is invertible since $V$ is nowhere vanishing. On $1$-forms, it is given by $d^{\nabla}_{\Ss^1}\iota_V$ where again each map is invertible. We can then write
$    u=\iota_V (\mathcal{L}_V^{\nabla})^{-1}f,$
which solves
\begin{equation*}
    d^{\nabla_\rho} u =d^{\nabla_\rho}\iota_V (\mathcal{L}_V^{\nabla})^{-1}f=(\mathcal{L}_V^{\nabla}-\iota_V d^{\nabla_\rho})(\mathcal{L}_V^{\nabla})^{-1}f=f-(\mathcal{L}_V^{\nabla})^{-1}\iota_V d^{\nabla_\rho}f=f.
\end{equation*}
This shows that $\rho$ is acyclic.
\end{proof}

To conclude this section, we recall an important result on the representation variety $\mathrm{Hom}(\pi_1(M), \mathrm{GL}_r(\CC))$, which follows from the arguments in \cite{RBC}. Recall that $\Hom_{\mathrm{irr}}(\pi_1(M), \mathrm{GL}_r(\CC))$ denotes the subset of irreducible representations.

\begin{prop}\label{prop:connected2}
There exists a Zariski open subset $\mathcal{V}$ in $\Hom_{\mathrm{irr}}(\pi_1(M), \mathrm{GL}_r(\CC))$, such that the following properties hold:
\begin{enumerate}
    \item $\mathcal{V}=\bigcup_{j=1}^{r(2G-2)}\mathcal{V}_j$ where $\mathcal{V}_j=\{\rho\in \mathcal{V}: \rho(c)=e^{2\pi i j/(r(2G-2))}\}$.
    \item Each $\mathcal{V}_j$ is a smooth and path-connected algebraic variety of dimension $(2G-1)r^2+1$.
    \item Each $\mathcal{V}_j$ contains a unitary representation.
    \item $\dim_{\CC} (\Hom_{\mathrm{irr}}(\pi_1(M), \mathrm{GL}_r(\CC))\setminus \mathcal{V})\leq (2G-1)r^2$.
\end{enumerate}
\end{prop}
\begin{proof}For a matrix $M\in \mathrm{M}_{r\times r}(\CC)$, we denote by $Z(M):=\{N\in \mathrm{M}_{r\times r}(\CC)\mid [M,N]=0\}$ its centralizer. We will say that $M$ is regular if $Z(M)$ has dimension $r$ (equivalently, $M$ has $r$ distinct eigenvalues). 
 We define
\begin{equation*}
\begin{split}
    \mathcal{V}:=\{&\rho\in \Hom_{\mathrm{irr}}(\pi_1(M), \mathrm{GL}_r(\CC)): \text{ there exists } \ell \text{ such that } \\
    &\rho(a_{\ell}),\, \rho(b_{\ell}) \text{ are regular and } Z(\rho(a_{\ell}))\cap Z(\rho(b_{\ell})) = \CC \id\}.
\end{split}
\end{equation*}
Since the complement of $\mathcal{V}$ is given by some algebraic relations, we know that $\mathcal{V}$ is Zariski open. We now show the four properties of $\mathcal V$ claimed in the proposition.
\begin{enumerate}
    \item Let $\rho\in \mathcal V$. Since $\rho(c)$ commutes with all the elements in $\rho(\pi_1(M))$, by the definition of $\mathcal{V}$, $\rho(c)$ must be a scalar matrix. Since $\det(\rho(c))=1$, we know $\rho(c)=e^{2\pi i j /(r(2G-2))}$ for some~$j$.
    \item The smoothness comes from the submersion theorem. By \cite[Lemma 6]{RBC}, the differential of the map
    \begin{equation*}
        (A,B)\mapsto [A,B] : \mathrm{GL}_r(\CC)\times \mathrm{GL}_r(\CC)\to \mathrm{SL}_r(\CC)
    \end{equation*}
    is surjective at $(A,B)$ if $A,B$ are regular and $Z(A)\cap Z(B)=\CC \id$.
    Therefore, the following map
    \begin{equation*}
        (A_1,\cdots,A_G,B_1,\cdots,B_G)\mapsto [A_1,B_1]\cdots [A_G,B_G]:\mathrm{GL}_r(\mathbb{C})^{2G}\to \mathrm{SL}_r(\CC)
    \end{equation*}
    has surjective differential at $A_i=\rho(a_i)$ and $B_i=\rho(b_i)$ in $\mathcal{V}$. By the submersion theorem, we conclude $\mathcal{V}$ is smooth with dimension $2G r^2-(r^2-1)=(2G-1)r^2+1$. In order to show the connectedness, we note that it suffices to show that
    \begin{equation*}
    \begin{split}
     \mathcal{V}_j^{\ell}:=\{&\rho\in \Hom(\pi_1(M), \mathrm{GL}_r(\CC)): 
    \rho(a_{\ell}),\, \rho(b_{\ell}) \text{ are regular and }\\
    &Z(\rho(a_{\ell}))\cap Z(\rho(b_{\ell})) = \CC \id,\,\, \rho(c)=e^{2\pi i j/(r(2G-2))}\}
    \end{split}
    \end{equation*}
    is connected, since they have a nontrivial intersection for different $\ell$. Consider the projection $p: \mathcal{V}_j^{\ell}\to \mathrm{GL}_r(\CC)^{2G-2}$ to the components not equal to $\rho(a_{\ell}),\rho(b_{\ell})$. By \cite[Proposition 5]{RBC}, for a Zariski open set $\mathcal{C}$ of $C\in \mathrm{SL}_r(\CC)$, the variety
    \begin{equation*}
        \{(A,B)\in \mathrm{GL}_r(\CC)^2:[A,B]=C \text{ and } A,B \text{ regular with } Z(A)\cap Z(B) =\CC\id\}
    \end{equation*}
    is connected. Therefore, the fibers over a Zariski open subset $\mathcal{D}$ in $\mathrm{GL}_r(\CC)^{2G-2}$ are connected. Since $\mathrm{GL}_r(\CC)^{2G-2}$ is irreducible, $\mathcal{D}$ is connected and $p^{-1}(\mathcal{D})$ is connected. Since $p^{-1}(\mathcal{D})$ is also dense, this implies the connectedness of $\mathcal{V}_j^{\ell}$. 
    \item First we find $A,B\in U(r)$ such that
    \begin{equation*}
        [A,B]=ABA^{-1}B^{-1} = \omega \, \id_r,\quad \omega=e^{2\pi i/r}.
    \end{equation*}
    We can just take
    \begin{equation*}
        A=\begin{pmatrix}
            0&1&0&\cdots&0\\
            0&0&1&\cdots&0\\
            \vdots&\vdots&\vdots&\ddots&\vdots\\
            0&0&0&\cdots&1\\
            1&0&0&\cdots&0
        \end{pmatrix},\quad B=\begin{pmatrix}
            1&0&0&\cdots&0\\
            0&\omega&0&\cdots&0\\
            0&0&\omega^2&\cdots&0\\
            \vdots&\vdots&\vdots&\ddots&\vdots\\
            0&0&0&\cdots&\omega^{r-1}
        \end{pmatrix}.
    \end{equation*}
It is direct to check that $A,B$ are regular and $Z(A)\cap Z(B)=\CC \id$. Then we can similarly construct $A_2,B_2\in U(r)$ such that
\begin{equation*}
        [A_2,B_2]=A_2B_2A_2^{-1}B_2^{-1} = \omega^{-1}\zeta \, \id_r,\quad \zeta=e^{2\pi i j/r}.
\end{equation*}
We can define a representation by 
\begin{equation*}
\begin{split}
    &\rho(a_1)=A, \,\, \rho(b_1)=B, \,\, \rho(a_2)=A_2, \,\, \rho(b_2)=B_2,\\
    &\rho(c)=e^{2\pi i j/((2G-2)r)}\id,\quad \rho(a_i)=\rho(b_i)=\id,\quad i>2.
\end{split}
\end{equation*}
This gives a unitary representation in $\mathcal{V}_j$.
\item We will assume $r\geq 2$ since the case $r=1$ is obvious. Let $\rho \in \Hom_{\mathrm{irr}}(\pi_1(M), \mathrm{GL}_r(\CC))$. By Schur's lemma and irreducibility of $\rho$, we have $\rho(c)\in \CC\id$. The equation \eqref{eq:pi1pre} implies $\rho(c) = e^{2\pi i j/(r(2G-2))} \id$ for some $j$. Suppose $\mathcal{Z}$ is an irreducible component of 
\[\{\rho\in \Hom_{\mathrm{irr}}(\pi_1(M), \mathrm{GL}_r(\CC)): \rho(c)=e^{2\pi i j/(r(2G-2))}\}\]
that does not intersect $\mathcal{V}_j$ (we are done if such $\mathcal{Z}$ does not exist), then we can define the Zariski open subsets
\begin{equation*}
    \begin{split}
        \mathcal{Z}_1&= \{\rho\in \mathcal{Z}: \rho(a_1),\rho(b_1) \text{ are regular}\},\\
        \mathcal{Z}_2&= \{\rho\in \mathcal{Z}: Z(\rho(a_1))\cap Z(\rho(b_1)) = \CC \id\}.
    \end{split}
\end{equation*}
We claim $\mathcal{Z}_1$ is not empty. Suppose $\rho_0\in \mathcal{Z}$ does not belong to any other irreducible component. Let ${W}_0$ be an irreducible component of 
\begin{equation*}
    \{(A,B)\in \mathrm{GL}_r(\CC): [A,B]=e^{2\pi i j/r}[\rho_0(a_G),\rho_0(b_G)]^{-1}\cdots[\rho_0(a_2),\rho_0(b_2)]^{-1}\}.
\end{equation*}
By \cite[\S 1]{RBC}, $W_0$ contains $(A_1,B_1)$ that are regular. Thus $\mathcal{Z}_1$ contains an element with 
\begin{equation*}
    \rho(a_1)=A_1,\,\, \rho(b_1)=B_1, \,\, \rho(a_j)=\rho_0(a_j),\,\, \rho(b_j)=\rho_0(b_j),\quad j\geq 2
\end{equation*}
and is nonempty. Since $\mathcal Z$ is irreducible, it is connected and thus $\mathcal{Z}_2$ must be empty, i.e., $\dim_{\CC}(Z(\rho(a_1))\cap Z( \rho(b_1)))>1$ for any $\rho \in \mathcal{Z}$. Now we consider the map
\begin{equation*}
    p: \mathcal{Z}\to (\mathrm{GL}_r(\CC))^{2G-2},\quad p(\rho)=(\rho(a_1),\rho(b_1),\rho(a_3),\rho(b_3),\cdots, \rho(a_G),\rho(b_G)).
\end{equation*}
Let \[\mathcal{W}_0:=\{(A,B)\in \mathrm{GL}_r(\CC)\times \mathrm{GL}_r(\CC) : \dim_{\CC}(Z(A)\cap Z(B)) >1\}.\]
By \cite[Lemma 8 (iii)]{RBC}, $\dim_{\CC}(\mathcal{W}_0)\leq 2r^2-2(r-1)$ and thus
\begin{equation*}
    \dim_{\CC}(p(\mathcal{Z}))\leq \dim_{\CC}(\mathcal{W}_0\times (\mathrm{GL}_r(\CC))^{2G-4})\leq (2G-2)r^2-2(r-1).
\end{equation*}
Each fiber of $p$ is a subvariety of
\begin{equation*}
    \begin{split}
        \mathcal{W}_\rho&:=\{(A,B)\in \mathrm{GL}_r(\CC)\times \mathrm{GL}_r(\CC): [A,B]=z\},\\
        &z=e^{2\pi i j/r}[\rho(a_1),\rho(b_1)]^{-1}[\rho(a_G),\rho(b_G)]^{-1}\cdots[\rho(a_3),\rho(b_3)]^{-1}.
    \end{split}
\end{equation*}
By \cite[Lemma 8 (ii)]{RBC}, $\dim_{\CC}(\mathcal{W}_{\rho})\leq r^2+r$. Therefore,
\begin{equation*}
    \dim_{\CC}(\mathcal{Z})\leq  \dim_{\CC}(p(\mathcal{Z}))+r^2+r \leq (2G-1)r^2-r+2\leq (2G-1)r^2
\end{equation*}
since we assumed $r\geq 2$.\qedhere
\end{enumerate}
\end{proof}

\subsection{Twisted Pollicott--Ruelle resonances.}
\label{sec:Resonance}
In this section, we consider an oriented closed Riemannian manifold $(\Sigma,g)$ of dimension $n+1\geq 2.$ We assume the geodesic flow $(\varphi_t^g)_{t\in \mathbb R}$ on the unit tangent bundle $M$ is an \emph{Anosov flow}, which means that if we denote by $X=\tfrac{d}{dt}|_{t=0}\varphi_t^g$ the geodesic vector field, then there exists a flow-invariant, continuous splitting $TM=E_u\oplus \mathbb R X\oplus E_s$ of the tangent bundle and $C,\theta>0$ such that for any $p\in M$, 
\begin{align*}
&\|d\varphi_t^g (p)v_s\|\leq Ce^{-\theta t}\|v_s\|, \ v_s\in E_s(p), \  t\geq 0,
\\&\|d\varphi_t^g (p)v_u\|\leq Ce^{-\theta |t|}\|v_u\|, \ v_u\in E_u(p), \ t\leq 0.
\end{align*}
The bundle $E_s$ (resp. $E_u$) is the \emph{stable} (resp. \emph{unstable}) bundle of $(\varphi_t^g)_{t\in \mathbb R}.$ The Anosov splitting comes with a dual splitting on the cotangent bundle of $M$:
\begin{align*}T^*M=E_u^*\oplus \mathbb R \alpha \oplus E_s^*,\quad E_u^*(E_u\oplus \mathbb RX)=0, \ E_s^*(E_s\oplus \mathbb RX)=0,
\end{align*}
and where $\alpha\in C^{\infty}(M;T^*M)$ is the contact $1$-form on $M$ which satisfies $\alpha(X)\equiv 1$ and $d\alpha(X,\cdot)\equiv 0$.

 For a closed geodesic $\gamma$, we will write $[\gamma]$ for the class in $\pi_1(M)$ representing it. Recall that there is a $C>0$ such that for any $T>0,$ the number of closed geodesics of length less than $T$ is bounded by $Ce^{CT}$. This justifies that \eqref{eq:zeta} converges for $s\in \CC$ with $\mathrm{Re}(s)\gg 1$.

Let $\rho \in \mathrm{Hom}(\pi_1(M), \mathrm{GL}_r(\CC))$ and let $\mathcal E_{\rho}$ denote the associated flat vector bundle with differential $d^{\nabla_\rho}$. Define the graded vector bundle
\begin{equation}
\label{eq:forms}\mathscr E_\rho:=\bigoplus_{k=0}^{2n+1} \mathscr E^k_\rho,\quad \mathscr E^k_\rho=\Lambda^kT^*M\otimes \mathcal E_\rho, \quad \mathscr E^k_{\rho,0}:=\mathscr E^k_\rho\cap \mathrm{ker}(\iota_X). 
\end{equation}
For $k=0,1,2,\ldots, 2n$, the generator $X$ acts naturally on the smooth vector bundle $C^{\infty}(M; \mathscr E^k_{\rho,0})$.  The action is given by $\mathcal L_{X^\rho}:=\iota_X d^{\nabla_\rho}+d^{\nabla_\rho} \iota_X$.

One can associate to  the action of $\mathcal L_{X^\rho}$ on $C^{\infty}(M, \mathscr E^k_{\rho,0})$ a discrete spectrum $\mathrm{Res}^{k}(X,\rho)\subset \mathbb C$, the \emph{Pollicott--Ruelle resonances} by making it act on \emph{anisotropic spaces}, see for instance \cite{BKL,Bl,BT,GouLi,Fau08,Fau10,zeta,DG16}.

More precisely, from \cite{Fau10} (see also \cite{DG16} for the extension to smooth vector bundles), the resolvent $R_k^\rho(\lambda):L^2(M; \mathscr E^k_{\rho,0})\to L^2(M; \mathscr E^k_{\rho,0})$, defined for $\lambda\in \mathbb C$ with $\mathrm{Re}(\lambda)\gg 1$ admits a meromorphic extension $R_k^\rho(\lambda)$ to $\mathbb C$
$$R_k^\rho(\lambda):C^{\infty}(M; \mathscr E^k_{\rho,0})\to \mathcal D'(M; \mathscr E^k_{\rho,0}). $$
The poles of the extension are intrinsic and are called the resonances of $X$ acting on $\mathscr E^k_{\rho,0}$. The set of resonances is denoted by $\mathrm{Res}^k(X,\rho)$ or $\mathrm{Res}^k(\rho)$ if there is no possible confusion on the metric. A complex number  $\lambda_0\in \mathbb C$ is in $ \mathrm{Res}^k(\rho)$, if and only if
\begin{equation}
\label{eq:res}
    \Res_0^{k,1}(\rho,\lambda_0)=\{u\in \mathscr{D}'(M;\mathscr E^k_{\rho,0}): \mathcal{L}_{X^\rho} u = \lambda_0 u , \, \WF(u)\subset E_u^* \}\neq \{0\},
\end{equation}
where $\mathrm{WF}(u)$ denotes the wavefront set of a distributional current $u$, see \cite[Chapter 3]{Hor}. The non-zero elements of $\Res_0^{k,1}(\rho,\lambda_0)$ are called \emph{resonant states} at $\lambda_0$. The spectral projector at $\lambda_0$ is given by
\begin{equation}
\label{eq:Pi_k}\Pi^k_\rho(\lambda_0)=\frac{1}{2i\pi}\int_{\gamma}R_k^\rho(z)dz,
\end{equation}
where $\gamma$ is a small loop around $\lambda_0$. The (algebraic) multiplicity of $\lambda_0$ is given by the rank of the spectral projector.  The \emph{generalized resonant states} are the elements in the range of $\Pi^k_\rho(\lambda_0)$:
\begin{equation}
\label{eq:genres}
\Res_0^{k,\infty}(\rho,\lambda_0)=\{u\in \mathscr{D}'(M;\mathscr E^k_{\rho,0}): \exists \ell\in \mathbb{N},\ (\mathcal{L}_{X^\rho}-\lambda_0)^{\ell} u = 0 , \, \WF(u)\subset E_u^* \}.
\end{equation}
We will say that the resonance $\lambda_0$ has \emph{no Jordan block} if $\Res_0^{k,\infty}(\rho,\lambda_0)=\Res_0^{k,1}(\rho,\lambda_0)$. We conclude this section by recalling the link between the zeta function $\zeta_{g,\rho}(s)$ and the generalized resonant states. Define the \emph{dynamical determinant} of order $k$ to be
    \begin{equation}
        \label{eq:zeta_l}
\zeta_{k,\rho}(s)=\exp\left(-\sum_{\gamma\in \mathcal{P}}\sum_{j=1}^{+\infty}\frac 1 j \mathrm{tr}\big(\rho([\gamma])^j \big)\frac{\mathrm{tr}(\Lambda^k\mathcal P_{\gamma^j})}{|\mathrm{det}(\mathrm{Id}-\mathcal P_{\gamma^j})|}e^{-sj\ell_g(\gamma)}\right),
    \end{equation}
    where $\mathcal P_\gamma=d\varphi_{-\ell_g(\gamma)}|_{E_u\oplus E_s}(p)$, for a point $p$ on $\gamma$, is the Poincaré map along the geodesic $\gamma.$
    Recall that one has, see \cite[p. 20]{zeta} and \cite{DG16} for the extension to vector bundles,
\begin{equation}
    \label{eq:zerozetal}
    \forall s\in \mathbb C, \quad \mathcal \zeta_{\ell,\rho}(s)=0 \ \iff \ s\in \mathrm{Res}^\ell(\rho),
\end{equation}
and the multiplicity of the zero coincides with the algebraic multiplicity of $s$ as a resonance.

We can rewrite the twisted zeta function, for any $s\in \mathbb C$ such that $\mathrm{Re}(s)\gg 1,$
    \begin{align*}
\zeta_{g,\rho}(s)&=\prod_{\gamma\in \mathcal{P}} \det(\mathrm{Id}-\rho([\gamma])e^{-s\ell_g(\gamma)})=\exp\left(-\sum_{\gamma\in \mathcal{P}}\sum_{j=1}^{+\infty}\frac 1 j \mathrm{tr}\big(\rho([\gamma])^j \big)e^{-sj\ell_g(\gamma)}\right)
        \\&=\exp\left(-\sum_{\gamma\in \mathcal{P}}\sum_{j=1}^{+\infty}\frac 1 j \mathrm{tr}\big(\rho([\gamma])^j \big)\sum_{\ell=0}^{2n}(-1)^{\ell+1}\frac{\mathrm{tr}(\Lambda^\ell\mathcal P_{\gamma^j})}{|\mathrm{det}(\mathrm{Id}-\mathcal P_{\gamma^j})|}e^{-sj\ell_g(\gamma)}\right)=\prod_{\ell=0}^{2n}\zeta_{\ell,\rho}(s)^{(-1)^{\ell+1}},
    \end{align*}
    where we used the relation
    $$|\mathrm{det}(\mathrm{Id}-\mathcal P_{\gamma^j})|= \sum_{\ell=0}^{2n}(-1)^{\ell+1}\mathrm{tr}(\Lambda^\ell\mathcal P_{\gamma^j}).$$ The previous relation extends to $s\in \mathbb C$. In particular, using \eqref{eq:zerozetal} at $s=0$ yields \eqref{eq:orderofvanish}.
    \subsection{Co-resonant states and pairing}
    Let $\rho^*$ be the adjoint representation of $\rho$. We have a notion of \emph{(generalized) co-resonant state} dual to the (generalized) resonant states. A distribution is a (generalized) co-resonant state of $(X,\rho)$ if and only if it is a (generalized) resonant state of $(-X,\rho^*)$, in the sense explained above. Their set will be denoted by $\mathrm{Res}_0^{*,k,\bullet}(\rho,\lambda)$ for $\bullet \in \mathbb N\cup \{+\infty\}.$

There is a natural pairing $\langle \cdot,\cdot\rangle_{\mathcal E_\rho \times \mathcal E_{\rho^*}}$ between $\mathcal{E}_{\rho}$ and $\mathcal{E}_{\rho^*}$ given by the duality pairing on $\CC^r \times (\CC^r)^*$. Note that this indeed descends to a pairing to $\mathcal{E}_{\rho}\times \mathcal{E}_{\rho^*}$ since for any $ x\in\widetilde{M}$ and any $ \gamma \in \pi_1(M)$, one has
$$ \ \forall (v,w)\in \CC^r \times (\CC^r)^* ,\ \langle \rho(\gamma)v,\rho^*(\gamma)w\rangle=\langle \rho(\gamma)^{-1}\rho(\gamma)v,w\rangle=\langle v,w\rangle. $$
The pairing is non degenerate and extends to the bundle $\mathcal E_\rho \to M$ and to differential forms with values in the flat bundle. More precisely, we will work with the following non-degenerate pairing 
\begin{equation}
\label{eq:pairing}
(u,v)\in \mathscr{E}_{\rho,0}^k\times \mathscr{E}_{\rho^*,0}^{2n-k},\quad   h_{\alpha}(u,v)=\int_M u\wedge v\wedge \alpha.
\end{equation}

    Remark that the stable and unstable bundles of $-X$ and $X$ are flipped. In particular, the pairing $h_\alpha$ extends to $\mathrm{Res}_0^{*,k,\infty}(\rho,\lambda)\times \mathrm{Res}_0^{2n-k,\infty}(\rho,\lambda')$ for any $\lambda,\lambda'\in \CC$, by the wavefront conditions defining generalized resonant states, see \eqref{eq:genres}. By \cite[Equation (2.51)]{zeta3}, the pairing
\begin{equation}
    \label{eq:nondeg}
    h_\alpha: \Res_{0}^{k,\infty}(\rho,0)\times \Res_{0}^{*,2n-k,\infty}(\rho,0)\to \mathbb C
\end{equation}
is non-degenerate. For later use, we record that for $k=1,2$,
\begin{equation}
    \label{eq:adjointd}
    \forall (u,v) \in\Res_{0}^{k-1,\infty}(\rho,0)\times \Res_{0}^{*,2n-k,\infty}(\rho,0), \quad h_{\alpha}(d^{\nabla_\rho}u,v)=(-1)^k h_\alpha(u,d^{\nabla_{\rho^*}} v).
\end{equation}
Indeed, we compute directly using Stoke's theorem
\begin{align*}
 \quad h_{\alpha}(d^{\nabla_\rho}u,v)&=\int_M d^{\nabla_\rho}  u\wedge v\wedge \alpha=\int_M d (  u\wedge v\wedge \alpha)+ (-1)^k\int_M u\wedge d^{\nabla_{\rho^*}} v\wedge \alpha +\int_M u\wedge  v\wedge d\alpha
 \\&=(-1)^k\int_M u\wedge d^{\nabla_{\rho^*}} v\wedge \alpha=(-1)^k h_\alpha(u,d^{\nabla_{\rho^*}} v),
\end{align*}
where we used \eqref{eq:Leib}, $\iota_Xu=0$ and $\iota_Xv=0.$
\section{Perturbation theory}
\label{sec:pert}
In order to prove Theorems \ref{maintheo} and \ref{theo2}, we will use some perturbation results. Let $(\Sigma,g)$ be a closed Anosov surface and let $r\in \mathbb N$. We will apply perturbation theory to the twisted Pollicott--Ruelle resonances. Before that, we will need to identify the flat bundles $\mathcal E_\rho$ for different representations $\rho \in \mathrm{Hom}(\pi_1(M), \mathrm{GL}_r(\CC)). $  Actually, we show in the next lemma that the identification can be chosen (locally) to depend analytically in $\rho$, which will be important in the arguments of \S \ref{sec:proofMaintheo}. To state the result, we consider a compactly supported function $\chi\in C_c^{\infty}(\widetilde{M},\mathbb R_+)$ such that
\begin{equation}
    \label{eq:chi}
    \forall x\in \widetilde{M},\quad \sum_{\gamma \in \pi_1(M)}\chi(\gamma\cdot x)=1.
\end{equation}
\begin{lem}
\label{lem:ident}
    Let $\rho_0 \in \mathrm{Hom}(\pi_1(M), \mathrm{GL}_r(\CC)).$ There is a neighborhood $ \mathcal U$ of $\rho_0$ such that the map $F_{\rho}: \widetilde{M}\times \CC^r\to \widetilde{M}\times \CC^r$, where 
    \begin{equation}
        \label{eq:f_rho}
   \forall (x,v)\in \widetilde{M}\times \CC^r, \ F_{\rho}(x,v)=(x,f_\rho(x)v),\quad f_\rho(x)=\sum_{\gamma \in \pi_1(M)}\chi(\gamma\cdot x)\rho(\gamma)^{-1}\rho_0(\gamma)
    \end{equation}
    descends to an isomorphism of smooth bundles between $\mathcal E_{\rho_0}$ and $\mathcal E_\rho$. Moreover, the isomorphism $F_\rho$ depends analytically on $\rho \in \mathcal U.$
\end{lem}
\begin{proof}
    Let $(x,v)\in \widetilde{M}\times \CC^r $ and let $\gamma_0 \in \pi_1(M)$. We first check that the map $F_\rho$ descends to the flat bundles, i.e., we show $F_\rho(x,v)\sim_{\rho}F_\rho(\gamma_0\cdot x,\rho_0(\gamma_0)v).$
    In other words, we want to verify that $(x, f_\rho(x)v)\sim_{\rho}(\gamma_0\cdot x,f_\rho(\gamma_0 \cdot x)\rho_0(\gamma_0)v)$. This is true since 
    \begin{align*}
        \rho(\gamma_0)f_\rho(x)v&=\rho(\gamma_0)\sum_{\gamma \in \pi_1(M)}\chi(\gamma\cdot x)\rho(\gamma)^{-1}\rho_0(\gamma)v=\sum_{\gamma \in \pi_1(M)}\chi(\gamma\cdot x)\rho(\gamma\gamma_0^{-1})^{-1}\rho_0(\gamma)v
        \\&=\sum_{\gamma' \in \pi_1(M)}\chi(\gamma'\cdot \gamma_0\cdot x)\rho(\gamma')^{-1}\rho_0(\gamma'\gamma_0)v=f_\rho(\gamma_0 \cdot x)\rho_0(\gamma_0)v.
    \end{align*}
    This shows that the map $F_\rho$ descends to map $F_\rho: \mathcal E_{\rho_0}\to \mathcal E_{\rho}$ which preserves the fibers. It is clearly a smooth map. Since $f_{\rho_0}\equiv \mathrm{Id}$ and since $\rho \mapsto f_\rho$ (and thus $F_\rho$) is an analytic map in $\rho$, there is a small neighborhood $\mathcal U$ such that for any $\rho \in \mathcal U$ and for any $x\in \widetilde{M},$ one has $\mathrm{det}(f_\rho(x))\neq 0$. This shows that $F_\rho$ is an isomorphism for $\rho \in \mathcal U$. 
\end{proof}
Let $\rho_0 \in \mathrm{Hom}(\pi_1(M), \mathrm{GL}_r(\CC))$ and let $\mathcal U$ be the neighborhood of $\rho_0$ obtained in Lemma \ref{lem:ident}. For any $\rho \in \mathcal U$, the Lie derivative $\mathcal L_{X^\rho}$ pullbacks to an operator on $\mathcal{E}_{\rho_0}$ we will denote by $\widetilde{\mathcal L}_{X^\rho}$. Its resolvent on $k$-forms will be denoted by  $\widetilde{R}_k^\rho(\lambda)$ and for a Pollicott--Ruelle resonance $\lambda_0 \in \mathrm{Res}^k(\rho_0)$, we can define a spectral projector $\widetilde{\Pi^k_\rho}(\lambda_0)$ by integrating on a small enough loop around $\lambda_0$, as in \eqref{eq:Pi_k}. The small loop includes all the resonances of $\widetilde{R}_k^{\rho}(\lambda)$ close to~$\lambda_0$.
\begin{prop}
\label{prop:pert}
    Let $\rho_0 \in \mathrm{Hom}(\pi_1(M), \mathrm{GL}_r(\CC))$. Let $k=0,1,2$ and $\lambda_0 \in \mathrm{Res}^k(\rho_0)$ be a Pollicott--Ruelle resonance of $\rho_0$. Then there exists a neighborhood $ \mathcal U$ of $\rho_0$ such that
       $ (\widetilde{\Pi^k_\rho}(\lambda_0))_{\rho \in \mathcal U}$ is a family of operators of rank equal to $\mathrm{dim}(\mathrm{Res}_0^{k,\infty}(\rho_0,\lambda_0))$ which depends analytically on $\rho \in \mathcal U$.
\end{prop}

\begin{proof}
    Using the identification of Lemma~\ref{lem:ident}, we can use the theory of perturbation of Pollicott--Ruelle resonances due to Bonthonneau \cite{Bon}, see also the related work of Dang, Guillarmou, Rivière and Shen \cite{DGRS}. Indeed, from the construction of the flat connection recalled in \S \ref{sec:repr} and Lemma \ref{lem:ident}, we see that $\mathcal U \ni\rho\mapsto \widetilde{\mathcal L}_{X^\rho}$ depends analytically on $\rho$. Using uniform anisotropic spaces, the resolvent $\widetilde{R}_k^\rho(\lambda)$ (and thus the spectral projector by \eqref{eq:Pi_k}) is seen to depend analytically on $\rho$. This also shows that $(\widetilde{\Pi^k_\rho}(\lambda_0))_{\rho \in \mathcal U}$ is a family of operators of rank equal to $\mathrm{rk}(\Pi^k_{\rho_0}(\lambda_0))=\mathrm{dim}(\mathrm{Res}_0^{k,\infty}(\rho_0,\lambda_0))$.
\end{proof}

    \section{Order of vanishing when there is no Jordan block}
    \label{sec:ofv}
In this section, we compute $m(g,\rho)$ for representations with no Jordan block at zero on surfaces. Most of the arguments are adapted from \cite{zazi}.  
\begin{prop}
\label{prop:noJD}
    Let $(\Sigma,g)$ be an orientable connected closed Anosov surface of genus $G\geq 2$ and let $\rho \in \mathrm{Hom}(\pi_1(M), \mathrm{GL}_r(\CC))$ be such that there is no Jordan block at zero for $k=0,1,2$. Then
     \begin{enumerate}
        \item $m(g,\rho)=\dim(\rho)(2G-2)$ if $\pi$ factors through $\pi_1(\Sigma)$;
        \item $m(g,\rho)=0$ if $\pi$ does not factor through $\pi_1(\Sigma)$ and is furthermore irreducible.
    \end{enumerate}
\end{prop}
\begin{proof} 

In contrast to \cite[Proposition 3.1]{zazi}, we will not compute each dimension separately. The strategy consists in considering a linear operator $T$ and computing the alternate difference \eqref{eq:orderofvanish} from the dimension of the kernel of $T.$
The operator we will consider is  
\begin{equation}
    \label{eq:T}
    T:\Res_{0}^{1,1}(\rho,0) \to \Res_{0}^{2,1}(\rho,0), \quad u\mapsto d^{\nabla_\rho} u. 
\end{equation}
Note that the map $T$ is well-defined. Indeed, let $u\in \Res_{0}^{1,1}(\rho,0)$. Since $[d^{\nabla_\rho},\mathcal L_{X^\rho}]=0,$ one has $\mathcal L_{X^\rho} T(u)=0.$ Next, we have $\mathrm{WF}(d^{\nabla_\rho} u)\subset \mathrm{WF}(u)\subset E_u^*$ by general properties of the wavefront set.
Finally, using $\mathcal{L}_{X^\rho}= d^{\nabla_\rho}\iota_X+\iota_X d^{\nabla_\rho}$, one computes
$$ \iota_X T(u)=\iota_X d^{\nabla_\rho} u=\mathcal L_{X^\rho} u- d^{\nabla_\rho}\iota_X u=0.$$
Note that in this last computation, we use the fact that $u$ is a resonant state (and not a generalized one). We show the following result.
\begin{lem}
\label{lemm:1}
  One has
    \begin{equation}
        \label{eq:surj}
        \dim \Res_{0}^{1,1}(\rho,0)= \dim \ker T +\dim \Res_{0}^{2,1}(\rho,0)-\mathrm{dim}(\rho^{\pi_1(M)}).
    \end{equation}
\end{lem}
In particular, if $\rho$ does not have any trivial factor, that is if $\rho^{\pi_1(M)}=\{0\},$  then $T$ is a surjective map.
\begin{proof}
    Since there is no  Jordan block for $k=0,1,2$, \eqref{eq:nondeg} implies the non-degeneracy of the pairing 
    $$h_\alpha: \Res_{0}^{k,1}(\rho,0)\times \Res_{0}^{*,2-k,1}(\rho,0)\to \mathbb C.$$ 
    We will denote by $\mathrm{Ran}(T)^\perp:=\{u\in \mathrm{Res}^{*,0,1}_0(\rho,0)\mid \forall v\in \mathrm{Res}^{1,1}_0(\rho,0), \ h_\alpha(T(v),u)=0\}.$
    Using \eqref{eq:adjointd}, we obtain
    \begin{equation}
    \label{eq:equi}
    \begin{split}
        u\in \mathrm{Ran}(T)^\perp\ &\iff \ \forall v\in \mathrm{Res}^{1,1}_0(\rho,0), \  h_{\alpha}(d^{\nabla_\rho} v,u)=0 
        \\& \iff \ \forall v\in \mathrm{Res}^{1,1}_0(\rho,0), \  h_{\alpha}( v,d^{\nabla_{\rho^*}} u)=0 
        \\ &\iff \ d^{\nabla_{\rho^*}}u=0.
        \end{split}
    \end{equation}
    We study the kernel of $d^{\nabla_\rho}: \Res^{*,0,1}_{0}(\rho,0)\to \mathrm{Res}_0^{*,1,1}(\rho,0)$. Let $u\in \Res^{*,0,1}_{0}(\rho,0)$ such that $d^{\nabla_{\rho^*}} u=0.$ Since $d^{\nabla_{\rho^*}}$ is elliptic (it is clear in local coordinates), one has $u\in C^{\infty}(M;\mathcal E_{\rho^*})$. This means that, by Lemma \ref{lemm:coho}, $u\in H^0(M,\rho^*)=(\rho^*)^{\pi_1(M)}$ and this space is of dimension $\mathrm{dim}((\rho^*)^{\pi_1(M)})=\mathrm{dim}(\rho^{\pi_1(M)})$. The statement of the lemma follows from the rank-nullity theorem and the fact that $\mathrm{dim}(\mathrm{Ran}(T)^\perp)=\mathrm{codim}(\mathrm{Ran}(T)).$ 
    \end{proof}
    Note that for any $(v,w)\in \rho^{\pi_1(M)}\times (\rho^*)^{\pi_1(M)}$, one has $vd\alpha \in \mathrm{Res}_0^{2,1}(\rho,0)$ and 
    $$h_\alpha(vd\alpha,w)=\int_M \langle v,w\rangle_{\mathcal E_\rho \times \mathcal E_{\rho^*}} d\alpha\wedge\alpha=\langle v,w\rangle_{\mathcal E_\rho \times \mathcal E_{\rho^*}}.  $$
    Since the pairing $\rho^{\pi_1(M)}\times (\rho^*)^{\pi_1(M)}\ni(v,w)\mapsto \langle v,w\rangle_{\mathcal E_\rho \times \mathcal E_{\rho^*}}$ is non-degenerate, we deduce from \eqref{eq:equi} that $vd\alpha\notin \mathrm{Ran}(T)$. Since $\mathrm{codim}(\mathrm{Ran}(T))=\mathrm{dim}(\rho^{\pi_1(M)})$, this implies that 
    \begin{equation}
        \label{eq:Ran(T)}
\mathrm{Res}^{2,1}_0(\rho,0)=\mathrm{Ran}(T)\oplus \rho^{\pi_1(M)}d\alpha.
    \end{equation}
    We will need this decomposition of the resonant states in the next lemma.
\begin{lem}
\label{lemm:surj}
    One has
    \begin{equation}
        \label{eq:eq2}
       \dim \ker T =\dim H^1(M,\rho)+\dim \Res^{0,1}_0(\rho,0)-\mathrm{dim}(\rho^{\pi_1(M)}). 
    \end{equation}
\end{lem}
\begin{proof}
    Let $u\in \ker T$. Then applying \cite[Lemma 2.1]{zazi}, there is $v\in \mathcal D'_{E_u^*}(M;\mathcal{E}_\rho)\footnote{$\mathcal D'_{E_u^*}(M;\mathcal{E}_\rho)$ denotes the set of $u\in \mathcal D'(M;\mathcal{E}_\rho)$ such that $\mathrm{WF}(u)\subset E_u^*.$}$ such that
    $$u-d^{\nabla_\rho} v\in C^{\infty}(M,\mathscr{E}_\rho^1),\quad d^{\nabla_\rho} ( u-d^{\nabla_\rho} v)=0.$$
    In particular, one can define the mapping
    \begin{equation}
        \label{eq:H1}
        S:\ker T\to H^1(M,\rho),\quad u\mapsto [u-d^{\nabla_\rho} v]_{H^1(M,\rho)}.
    \end{equation}
    The kernel of $S$ is given by $d^{\nabla_\rho} (\Res_0^{0,1}(\rho,0)).$ Indeed, the inclusion $d^{\nabla_\rho} (\Res_0^{0,1}(\rho,0))\subset \ker (S)$ is clear. Conversely, suppose that $S(u)=0$, i.e., $u-d^{\nabla_\rho} v$ is exact. By changing $v$ if necessary, we can suppose that $u-d^{\nabla_\rho} v=0.$ Applying $\iota_X$ gives $\mathcal L_{X^\rho} v=0.$ Since $v\in \mathcal D'_{E_u^*}(M,\mathcal{E}_\rho),$ this shows that $u\in d^{\nabla_\rho} (\Res_0^{0,1}(\rho,0)).$

    We now show that $S$ is surjective. Let $w\in C^{\infty}(M; \mathscr E^1_{\rho})$ and $d^{\nabla_\rho}w=0$. We need to find $v\in \mathcal D'_{E_u^*}(M;\mathcal E_{\rho})$ such that $w-d^{\nabla_\rho} v\in \Res_0^{1,1}(\rho,0).$ This is equivalent to $\iota_X(w-d^{\nabla_\rho} v)=0$ and thus to $ \mathcal L_{X^\rho}v=\iota_X w.$ This equation is solvable if one has
\begin{equation*}
   \forall y\in \Res_{0}^{*,2,1}(\rho,0),\quad h_{\alpha}(\iota_Xw,y)=\int_M \iota_X w \cdot  y \wedge \alpha=0.
\end{equation*}
We use \eqref{eq:Ran(T)} and first consider $y\in \mathrm{Ran}(T)$\footnote{More precisely, we apply the previous results to the adjoint representation $\rho^*.$}. That is, we consider $y=d^{\nabla_{\rho^*}} z$ for some $z\in \Res_0^{*,1,1}(\rho,0)$. Moreover, we have $d^{\nabla_\rho}\iota_X w= \mathcal{L}_{X^\rho} w$. In particular, using \eqref{eq:adjointd} gives
$$ h_{\alpha}(\iota_Xw,y)=h_{\alpha}(\iota_Xw,d^{\nabla_{\rho^*}} z)=-h_{\alpha}(\mathcal L_{X^\rho}w,z)=0.$$
Next, we consider $y\in (\rho^*)^{\pi_1(M)}d\alpha.$ Write $y=vd\alpha$ with $v\in (\rho^*)^{\pi_1(M)}$. We compute
\begin{align*}h_\alpha(\iota_X w,vd\alpha)&=\int_M \langle w(X),v\rangle_{\mathcal E_\rho \times \mathcal E_{\rho^*}} d\alpha\wedge \alpha=\int_M d\alpha \wedge\langle w(\cdot),v\rangle_{\mathcal E_\rho \times \mathcal E_{\rho^*}}
\\&=\int_M \alpha \wedge \langle d^{\nabla_\rho}w(\cdot,\cdot), v\rangle_{\mathcal E_\rho \times \mathcal E_{\rho^*}}=0. 
\end{align*}
This shows that $S$ is surjective.
To conclude the proof of the Lemma, we apply the rank-nullity theorem to the map $S$ and to $d^{\nabla_\rho}:\Res_0^{0,1}(\rho,0)\to \Res_0^{1,1}(\rho,0)$ to obtain
\begin{equation*}
\begin{split}
\dim \ker T &=\dim H^1(M,\rho)+\mathrm{dim}(d^{\nabla_\rho} (\Res_0^{0,1}(\rho,0)))
\\&=\dim H^1(M,\rho)+\mathrm{dim}(\Res_0^{0,1}(\rho,0))-\mathrm{dim}(\mathrm{ker}(d^{\nabla_\rho}))
\\&=\dim H^1(M,\rho)+\mathrm{dim}(\Res_0^{0,1}(\rho,0))-\mathrm{dim}(\rho^{\pi_1(M)}).\qedhere
\end{split}
\end{equation*}
\end{proof}
Combining \eqref{eq:surj} and \eqref{eq:eq2}, we deduce
\begin{equation*}
    \dim \Res_{0}^{1,1}(\rho,0)-\dim \Res_{0}^{0,1}(\rho,0)-\dim \Res_{0}^{2,1}(\rho,0)=\dim H^1(M,\rho)-2\mathrm{dim}(\rho^{\pi_1(M)}).
\end{equation*}
From Lemma \ref{lemm:coho}, we deduce that $m(g,\rho)=\dim(\rho)(2G-2)$ in Case \ref{case1} and $m(g,\rho)=0$ in Case \ref{case2}.
This concludes the proof of the proposition.
\end{proof}    

\begin{rem}
    Below is a slightly different argument for Proposition \ref{prop:noJD}, following an (unpublished) note of Dyatlov and Zworski (see also Dang--Rivi\`ere \cite{DR}). Suppose, as before, that there are no Jordan blocks, i.e.,
    $\Res_0^{k,\infty}(\rho,0)=\Res_0^{k,1}(\rho,0)$ for $k=0,1,2$.
    We can define a complex by
    \begin{equation*}
        d^{\nabla_\rho}:\mathcal{C}_0^k\to \mathcal{C}_0^{k+1},\quad \mathcal{C}_0^k=\Res_0^{k,1}(\rho,0).
    \end{equation*}
    We can also consider the following complex
    \begin{equation}
    \label{eq:complex}
        d^{\nabla_\rho}: \mathcal{C}^k\to \mathcal{C}^{k+1},\quad \mathcal{C}^k:=\{u\in \mathcal{D}_{E_u^*}'(M,\mathscr {E}_{\rho}^k):\mathcal{L}_{X^{\rho}} u =0\}, \quad k=0,1,2,3.
    \end{equation}
    We claim this complex is homotopic equivalent to the de Rham complex $\mathcal{D}_{E_u^*}'(M, \mathscr{E}_{\rho}^k)$.
    The chain maps between them are the inclusion $J_k:\mathcal C^k \hookrightarrow \mathcal D'_{E_u^*}(M,\mathscr E_\rho^k)$ and the spectral projection $\Pi_k$ at $0$. Recall that one has the following Laurent expansion near $z=0$ (we are using the semisimplicity here):
    \begin{equation*}
        (\mathcal{L}_{X^\rho}-z)^{-1}=H_k(z)-\frac{\Pi_k}{z}: \mathcal{D}_{E_u^*}'(M,\mathscr{E}_{\rho}^k)\to \mathcal{D}_{E_u^*}'(M,\mathscr{E}_{\rho}^k),
    \end{equation*}
    where $H_k(z)$ is holomorphic in $z$. Moreover, it is direct to check that
    \begin{equation*}
        \id_k-J_k\Pi_k=\mathcal{L}_{X^\rho}H_k(0)=\iota_X H_k(0)d^{\nabla_\rho}+d^{\nabla_\rho}\iota_X H_k(0).
    \end{equation*}
This gives an explicit homotopy equivalence between $\mathcal{C}^{\bullet}$ and the de Rham complex $\mathcal D'_{E_u^*}(M,\mathscr E_\rho^k)$.
For $k=0,1,2,3$, we have the short exact sequence:
\begin{equation*}
    0\to \mathcal{C}_0^k\to \mathcal{C}^k\to \mathcal{C}_0^{k-1}\to 0
\end{equation*}
where the second map is $\pi_k=(-1)^k\iota_X$. Let $\mathcal{H}^k$ be the $k$-th cohomology group of $\mathcal{C}$ and $\mathcal{H}_0^k$ be the $k$-th cohomology group of $\mathcal{C}_0$. We get the following long exact sequence:
\begin{equation*}
    \mathcal{H}_0^{k-2}\to\mathcal{H}_0^k\to \mathcal{H}^k\to \mathcal{H}_0^{k-1}\to \mathcal{H}_0^{k+1}\to \mathcal{H}^{k+1}\to\mathcal{H}_0^k\to \mathcal{H}_0^{k+2}.
\end{equation*}
For $k=0,$ this becomes
$$ 0\to\mathcal{H}_0^0\to \mathcal{H}^0\to 0\to \mathcal{H}_0^{1}\to \mathcal{H}^{1}\to\mathcal{H}_0^0\to \mathcal{H}_0^{2}$$
where the last map is $[u]\mapsto [ud\alpha]$.
The first part of the exact sequence implies
\begin{equation}\label{eq:exact}
    \mathcal{H}_0^0\cong \mathcal{H}^0,\quad \mathcal{H}_0^1\cong \ker(\pi_1^*: \mathcal{H}^1\to \mathcal{H}_0^0).
\end{equation}
From the rank-nullity theorem and the fact that the complex from \eqref{eq:complex} is homotopic equivalent to the de Rham complex, we obtain
\begin{align*}
    \dim \mathcal{H}_0^1&=\dim \ker(\pi_1^*: \mathcal{H}^1\to \mathcal{H}_0^0)=\dim H^1(M, \rho)-\dim \mathrm{Ran}(\pi_1^*: \mathcal{H}^1\to \mathcal{H}_0^0).
\end{align*}
From \eqref{eq:orderofvanish}, we obtain
\begin{equation*}
    m(g,\rho)=\dim\mathcal{C}_0^1-\dim \mathcal{C}_0^0-\dim \mathcal{C}_0^2= \dim \mathcal{H}_0^1-\dim \mathcal{H}_0^0 -\dim \mathcal{H}_0^2, 
\end{equation*}
where we use that the Euler characteristic of the complex $\mathcal{C}^{\bullet}$ is the same as its cohomology.

If we can show the map $[u]\mapsto [ud\alpha]:\mathcal{H}_0^0\to \mathcal{H}_0^2$ is bijective, then the exact sequence \eqref{eq:exact} gives $\mathrm{Ran}(\pi_1^*: \mathcal{H}^1\to \mathcal{H}_0^0)=0$,
and hence
\begin{equation*}
    m(g,\rho)=\mathcal{H}_0^1-2\dim \mathcal{H}_0^0 =\dim H^1(M,\rho)-2\dim H^0(M,\rho).
\end{equation*}
But the bijectivity follows from \eqref{eq:Ran(T)}.
\end{rem}

We record for later use the following consequence of the proof of Lemma \ref{lemm:surj}.
\begin{lem}
    \label{lemm:3}
     Let $(\Sigma,g)$ be an orientable connected closed Anosov surface of genus $G\geq 2$ with a  representation $\rho \in \mathrm{Hom}(\pi_1(M), \mathrm{GL}_r(\CC)).$  
    Assume that $\mathrm{Res}_0^{0,1}(\rho,0)=0$ and $\mathrm{Res}_0^{2,1}(\rho,0)=0$. Then
    $$m(g,\rho) \geq\mathrm{dim}(H^1(M,\rho)). $$
    Moreover, one has $m(g,\rho)=\mathrm{dim}(H^1(M,\rho))$ if and only if there is no Jordan block at zero for $k=1.$ Finally, the set of $\rho \in \mathcal V$ for which
    \begin{equation}
    \label{eq:fff}
          \mathrm{dim}(\Res_0^{k,\infty}(\rho,0))=0, \ k=0,2, \quad \mathrm{dim}(\Res_0^{1,\infty}(\rho,0))=\mathrm{dim}(\Res_0^{1,1}(\rho,0))=\dim(H^1(M,\rho))
         \end{equation}    
    is an open subset of the set $\mathcal V$ defined in Proposition \ref{prop:connected2}.
\end{lem}
\begin{proof}
    Since $\mathrm{Res}_0^{2,\infty}(\rho, 0)=0,$ the mapping
    \begin{equation*}
        T:\Res_{0}^{1,1}(\rho,0) \to \Res_{0}^{2,1}(\rho,0), \quad u\mapsto d^{\nabla_\rho} u
    \end{equation*}
    from \eqref{eq:T} is equal to $0.$ In particular, one can define the mapping $S$ from \eqref{eq:H1} for any $u\in \mathrm{Res}_0^{1,1}(\rho,0).$ The first part of the proof of Lemma \ref{lemm:surj} shows that $S$ is injective since $\mathrm{Res}_0^{0,1}(\rho,0)=0.$ The second part of the proof still shows that $S$ is surjective since there are no obstructions to solving $ \mathcal L_{X^\rho}(v)=\iota_X w.$ In particular, using \eqref{eq:orderofvan} and Lemma \ref{lemm:coho} gives
    \begin{equation}
    \label{eq:lowerbnd}
        m(g,\rho)=\mathrm{dim}(\mathrm{Res}^{1,\infty}_0(\rho,0))\geq \mathrm{dim}(\mathrm{Res}^{1,1}_0(\rho,0))=\mathrm{dim}(H^1(M,\rho)), 
        \end{equation}
    with equality if and only if $\mathrm{Res}_0^{1,1}(\rho,0)=\mathrm{Res}_0^{1,\infty}(\rho,0).$ For the last claim, notice that by Proposition \ref{prop:pert} the mapping
    $\rho\mapsto {\rm dim}(\mathrm{Res}_0^{k,\infty}(\rho,0))=\mathrm{dim}\big(\mathrm{ker}\big(\mathcal L_{X^\rho}\widetilde{\Pi^k_\rho}(\lambda_0)\big)\big) $ for any $k=0,1,2$ is upper semi-continuous. In particular, the conditions ${\rm dim}(\mathrm{Res}_0^{k,\infty}(\rho,0))=0$ for $k=0,2$ are open. Let $\rho_0\in \mathcal V$ be a representation satisfying \eqref{eq:fff}. For a close representation $\rho\in \mathcal V$, we have
    \begin{align*}
        \mathrm{dim}(H^1(M,\rho_0))&=\mathrm{dim}(\mathrm{Res}^{1,\infty}_0(\rho_0,0))\geq \mathrm{dim}(\mathrm{Res}^{1,\infty}_0(\rho,0))
        \\&\geq \mathrm{dim}(\mathrm{Res}^{1,1}_0(\rho,0))=\mathrm{dim}(H^1(M,\rho)).\end{align*}
    By Lemma \ref{lemm:coho}, one has $\mathrm{dim}(H^1(M,\rho_0))=\mathrm{dim}(H^1(M,\rho))$ for any close $\rho$ in $\mathcal V.$
    This means that the inequalities are in fact equalities and concludes the proof.
\end{proof}
To conclude this section, we compute the dimensions of the spaces of the generalized resonant states at zero for a unitary representation using Proposition \ref{prop:noJD} and \cite[Lemma 2.3]{zazi}. We note that in the case where $\rho$ factors through $\pi_1(\Sigma)$, this was obtained by Cekić and Paternain in \cite[Corollary 1.9]{CP20}. We provide a proof for a general representation $\rho$ of $\pi_1(M)$ for completeness. 
\begin{prop}
\label{prop:DZ}
    Let $(\Sigma, g)$ be an orientable connected closed Anosov surface and let $\rho$ be a unitary representation of $\pi_1(M)$ such that
    \begin{enumerate}
        \item either $\pi$ factors through $\pi_1(\Sigma)$ with no trivial factors, i.e., $\rho^{\pi_1(M)}=0$;
        \item or $\pi$ does not factor through $\pi_1(\Sigma)$ and is irreducible.
    \end{enumerate}Then
    $$\mathrm{dim}(\Res_0^{k,\infty}(\rho,0))=0, \ k=0,2, \quad \mathrm{dim}(\Res_0^{1,\infty}(\rho,0))=\mathrm{dim}(\Res_0^{1,1}(\rho,0))=\dim(H^1(M,\rho)). $$
\end{prop}
\begin{proof}
   Suppose that $k=0$ to start. Since $\rho$ is unitary, $\mathcal L_{X^\rho}$ acting on  $L^2(M, \mathcal E_{\rho,0}^0)$ is skew-adjoint. In particular, we can use \cite[Lemma 2.3]{zazi} to deduce that any $u\in \mathrm{Res}^{0,1}_0(\rho,0)$ is smooth.
 
 For any $x\in M$ and $v\in T_xM\otimes \mathcal{E}_{\rho}$, since $e^{t\mathcal{L}_X}u=u$, we have
   \begin{equation*}
       \langle d^{\nabla_\rho}u(x), v\rangle = \langle d^{\nabla_\rho}u(e^{t X}x), e^{t\mathcal{L}_{X^\rho}}v\rangle.
   \end{equation*}
   In particular, if $v\in E_s(x)\otimes \mathcal E_\rho$, since $d^{\nabla_\rho}u(e^{t X}x)$ is bounded, 
   this gives $\langle d^{\nabla_\rho}u(e^{t X}x), e^{t \mathcal L_{X^\rho}}v\rangle\to 0$ when $t\to+\infty$. We have a similar argument for the unstable bundle and this implies that $d^{\nabla_\rho}u(x)=\varphi \alpha$ for some $\varphi\in C^{\infty}(M;\mathcal{E}_{\rho})$. Since $$0=\alpha \wedge \underbrace{d^{\nabla_\rho}(\varphi\alpha)}_{d^{\nabla_\rho} d^{\nabla_\rho} u=0}=\varphi \alpha \wedge d\alpha,$$ this implies $\varphi=0$ and thus $d^{\nabla_\rho}u=0$. Lemma \ref{lemm:coho} implies $H^0(M,\rho)=0$ which gives $u\equiv 0.$ This also implies that $\mathrm{Res}^{2,1}_0(\rho,0)=0$ since $u\mapsto u d\alpha$ is an isomorphism from $\mathrm{Res}^{0,1}_0(\rho,0)$ to $\mathrm{Res}^{2,1}_0(\rho,0).$
   
   We now let $k=1$ and show that there are no Jordan block at zero, which will finish the proof of the proposition using Lemma \ref{lemm:3}. Suppose that $u \in \mathrm{Res}^{1,2}_0(\rho,0)$, that is $u\in \mathcal D'_{E_u^*}(M,\mathscr E^1_{\rho,0}) $ satisfying $\iota_X u=0$ and $\mathcal L_{X^\rho}u=\iota_Xd^{\nabla_\rho}u=:v\in \mathrm{Res}_0^{1,1}(\rho,0).$ We first notice that $\alpha \wedge d^{\nabla_\rho}u=a( \alpha\wedge d\alpha)$ for some $a\in \mathcal D'_{E_u^*}(M;\mathcal E_\rho)$. But since $[\mathcal L_{X^\rho},d^{\nabla_\rho}]=0$, we obtain
   \begin{align*}\mathcal L_{X^{\rho}}(a)(\alpha\wedge d\alpha)=\mathcal L_{X^\rho}(\alpha \wedge d^{\nabla_\rho}u)=\underbrace{\mathcal L_{X}\alpha}_{=0} \wedge d^{\nabla_\rho}u+ \alpha \wedge \mathcal L_{X^\rho}d^{\nabla_\rho}u =\alpha \wedge d^{\nabla_\rho}v=0,
   \end{align*}
     where we used that $d^{\nabla_\rho}v\in \mathrm{Res}^{2,1}_0(\rho, 0)=0$. In particular, this gives $a\in \mathrm{Res}^{0,1}_0(\rho, 0)=0$ and thus $\alpha \wedge d^{\nabla_\rho}u=0.$ We deduce that $d^{\nabla_\rho}u=\alpha \wedge v$ and from \cite[Lemma 2.1]{zazi}, there is
    $$\varphi\in \mathcal D'_{E_u^*}(M,\mathcal{E}_\rho), \ w \in C^{\infty}(M; \mathscr E^1_{\rho,0}),\quad v=w+d^{\nabla_\rho} \varphi,\quad d^{\nabla_\rho}w=0.$$
    Since $\iota_X v=0$, we obtain $\mathcal L_{X^\rho}(\varphi)=-\iota_X w$ and thus
    \begin{align*}
        0&=\int_M d^{\nabla_\rho} u\wedge \bar{w} =\mathrm{Re}\int_M d^{\nabla_\rho} u\wedge \bar{w}=\mathrm{Re}\int_M \alpha \wedge v\wedge \bar{w} =\mathrm{Re}\int_M \alpha \wedge d^{\nabla_{\rho}}\varphi\wedge \bar{w}
        \\&=\mathrm{Re}\int_M \varphi \bar{w}\wedge d\alpha=-\mathrm{Re}\langle \mathcal L_{X^\rho}(\varphi), \varphi\rangle.
    \end{align*}
    We can now finish the proof as in \cite[Lemma 3.5]{zazi} and use \cite[Lemma 2.3]{zazi} to conclude that $\varphi \in C^{\infty}(M;\mathcal E_{\rho})$ and thus $v \in C^{\infty}(M;\mathscr E^1_{\rho,0}).$ This readily implies $v=0$ by the same argument as for $k=0$. In other words, we showed that $\mathcal L_{X^\rho}u=0$ for any generalized resonant states at zero, i.e., that there is no Jordan block at zero for $k=1$ and this concludes the proof.
\end{proof}
\section{Proof or Theorem \ref{theo2}}
\label{sec:proofMaintheo}

In this section, we prove Theorem \ref{theo2}.
\begin{proof}
    We fix $j=1,\ldots, r(2G-2)$ and $\rho,\rho_0 \in \mathcal{V}_j\subseteq \mathrm{Hom}(\pi_1(M), \mathrm{GL}_r(\CC))$ with $\rho_0$ unitary and nontrivial. From Proposition \ref{prop:connected2}, we can consider an analytic path $\rho:[0,1]\to \mathcal{V}_j $ such that $\rho(0)=\rho_0$ and $\rho(1)=\rho.$
  Fix $t\in [0,1]$ and apply Proposition \ref{prop:pert} to $\rho(t)$. There is a neighborhood $\mathcal U_t$ of $\rho(t)$ such that the mapping
      $$\mathcal U_t\ni \tau \mapsto M_t^0(\tau):= \widetilde{\Pi_{\tau}^0}(0)  \widetilde{\mathcal L}_{X^\tau}\widetilde{\Pi_{\tau}^0}(0),   $$
      is an analytic family of finite rank operators. As a consequence, we deduce that the set $\{\tau \in \mathcal U_t \mid \mathrm{det}(M_t^0(\tau))=0\}$ is either equal to $\mathcal U_t$ or has complex codimension $\geq 1$. Remark that $\mathrm{det}(M_t^0(\tau))=0$ if and only if $0\in \mathrm{Res}^0(\tau).$
  Let
  $$I:=\{t\in [0,1]\mid \{\tau \in \mathcal U_t \mid \mathrm{det}(M_t^0(\tau))=0\}\neq \mathcal U_t\}. $$
   From Proposition \ref{prop:DZ}, we see that $0\in I$. Moreover, it is clear from the definition of $I$ that it is open. Suppose now that $s_0\notin I$, i.e., $\mathrm{det}(M_{s_0}(\tau))=0$ for any $\tau \in \mathcal U_{s_0}.$ Then for any $s \in \mathcal U_{s_0}$, we have $\mathrm{det}(M_{s_0}(\tau))=0$ for any $\tau \in \mathcal U_s\cap \mathcal U_{s_0},$ since the vanishing of the determinant is equivalent to the fact that $0$ is a resonance and hence does not depend on the representation $\rho(s)$ we fixed to apply Proposition \ref{prop:pert}. Since the determinant vanishes on an open subset of $\mathcal U_s$, it cannot vanish on a subset of codimension $\geq 1$ and we deduce that $s\in I$. This in turns imply that $[0,1]\setminus I$ is open. In particular, $I=[0,1]$ since $[0,1]$ is connected.
   In total, since resonant $0$-forms and $2$-forms can be identified, we have obtained the following.
   \begin{prop}
       \label{prop:k=0,2}
       Let $(\Sigma,g)$ be a closed Anosov surface and let $\rho \in \mathcal V_j$. Then there is a neighborhood $\mathcal U_\rho$ of $\rho $ in $\mathcal V_j$ and a complex codimension $\geq 1$ Zariski closed $\mathcal Z_\rho\subset \mathcal U_\rho$ such that 
$$\forall \tau \in \mathcal U_{\rho}\setminus \mathcal Z_{\rho},\quad \mathrm{dim}(\Res_0^{k,1}(\tau,0))=0, \ k=0,2. $$
   \end{prop}
   We now study resonant states for $k=1.$ Similarly, we fix $t\in [0,1]$ and define an analytic family of finite rank operators by
     $$ \mathcal U_t\ni \tau \mapsto M_t^1(\tau):= \widetilde{\Pi_{\tau}^1}(0)  \widetilde{\mathcal L}_{X^\tau}\widetilde{\Pi_{\tau}^1}(0),$$
For a matrix $A\in \mathrm{M}_{n\times n}(\mathbb R)$, denote by $\lambda_1,\ldots, \lambda_n\in \mathbb C$ its eigenvalues counted with algebraic multiplicity. For some $0\leq k\leq n$, let
$$\mathcal P_k(A):=\sum_{i_1<i_2<\ldots<i_k}\lambda_{i_1}\lambda_{i_2}\ldots \lambda_{i_k}. $$
\begin{lem}\label{lem:5.2}
    One has
    $\min\{J\mid \mathcal P_k(A)=0, \ \forall k>J\}=n-\mathrm{mult}_0(A), $
    where $\mathrm{mult}_0(A)$ denotes the algebraic multiplicity of $0$ as an eigenvalue of $M$.
\end{lem}
\begin{proof}
    Let $J=n-\mathrm{mult}_0(A)$ and up to reordering, let $\lambda_1=\ldots=\lambda_J$ denote the non-zero eigenvalues. Then
    $$\mathcal P_J(A)=\prod_{i=1}^J\lambda_i\neq 0. $$
    Moreover, if $k>J$, then any product in $\mathcal P_k(A)$ contains a zero and thus vanishes. 
\end{proof}
Note that by Newton's inversion formulas and since $\tau \mapsto \mathrm{tr}\big(M_t^1(\tau)^\ell\big)$ is analytic for any $\ell$, we obtain that $\tau \mapsto \mathcal P_\ell(M_t^1(\tau))$ is analytic. This in turn implies that the set  $\{\tau \in \mathcal U_t \mid \mathrm{mult}_0(M_t^1(\tau))\geq r(2G-2)+1\}$ is Zariski closed and by the same argument as before, we deduce that $\{\tau \in \mathcal U_t \mid \mathrm{mult}_0(M_t^1(\tau))\geq r(2G-2)+1\}$ has complex codimension $\geq 1$ for any $t\in[0,1].$ Recall that 
$$\mathrm{mult}_0(M_t^1(\tau))=\mathrm{dim}(\mathrm{Res}^{1,\infty}_0(\tau,0)). $$
Moreover, Proposition \ref{prop:k=0,2} and Lemma \ref{lemm:3} imply $\mathrm{dim}(\mathrm{Res}^{1,\infty}_0(\tau,0))\geq r(2G-2)$ with equality if and only if there is no Jordan block for $k=1.$ In total, we have shown the following result.
\begin{prop}
       \label{prop:k=1}
       Let $(\Sigma,g)$ be an orientable connected closed Anosov surface and let $\rho \in \mathcal{V}_j$. Then there is a neighborhood $\mathcal U'_\rho$ of $\rho $ in $\mathcal{V}_j$ and a codimension $\geq 1$ Zariski closed subset $\mathcal Z'_\rho\subset \mathcal U'_\rho$ such that 
$$\forall \tau \in \mathcal U'_\rho \setminus  \mathcal{Z}'_{\rho},\quad \mathrm{dim}(\Res_0^{1,\infty}(\tau,0))=\mathrm{dim}(\Res_0^{1,1}(\tau,0))=\dim(H^1(M,\tau)). $$
   \end{prop}
   We can now deduce Theorem \ref{theo2} from Propositions \ref{prop:k=0,2} and \ref{prop:k=1}. For $j=1,\ldots, r(2G-2)$, define $\mathcal U_{g,j}\subset \mathcal V_j$ by 
   $$ \mathcal U_{g,j}:=\bigcup_{\rho\in \mathcal V_j}(\mathcal U_\rho \setminus  \mathcal Z_{\rho})\cap\bigcup_{\rho\in \mathcal V_j}(\mathcal U'_\rho \setminus  \mathcal Z'_{\rho}).  $$
   The subset $\mathcal U_{g,j}$ is open in $\mathcal V_j$ since all $\mathcal U_\rho  \setminus \mathcal Z_{\rho}$ and $\mathcal U'_\rho \setminus  \mathcal Z'_{\rho}$ are open.  Its complement has complex codimension $\geq 1$. Moreover, if $\rho \in \mathcal U_{g,j}$, then there exists $\rho_1$ and $\rho_2$ in $\mathcal V_j$ such that $\rho \in \mathcal U_{\rho_1} \setminus  \mathcal Z_{\rho_1}$ and $\rho \in \mathcal U'_{\rho_2} \setminus  \mathcal Z_{\rho_2}'$. By Propositions \ref{prop:k=0,2} and \ref{prop:k=1}, this implies that 
   \begin{align*}
   \begin{split}
        &\mathrm{dim}(\Res_0^{k,\infty}(\rho,0))=0, \quad k=0,2, 
        \\ &\mathrm{dim}(\Res_0^{1,\infty}(\rho,0))=\mathrm{dim}(\Res_0^{1,1}(\rho,0))=\mathrm{dim}(H^1(M,\rho)). 
        \end{split}
        \end{align*}
        Theorem \ref{theo2} then follows from Lemma \ref{lemm:coho} by taking $\mathcal U_g:=\cup_{j=1}^{r(2G-2)}\mathcal U_{g,j}$.
\end{proof}

\subsection{Proof of Corollary \ref{cor}}
\begin{proof}
We first remark that the Reidemeister--Turaev torsion (see \cite[Theorem 2.3.1]{BFS23})
\begin{equation*}
    \tau_{\mathfrak{e}_{\mathrm{geod}},\mathfrak{o}}(\rho)= \det(\id-\rho(c))^{2G-2}
\end{equation*}
is a constant in $\mathcal{V}_j$ since $\rho(c)=e^{2\pi ij/(r(2G-2))}$, see Proposition \ref{prop:connected2}.

  Let $(\Sigma,g)$ be a closed negatively curved surface and let $\rho \in \mathcal U_g$ be a representation which does not factor through $\pi_1(\Sigma)$. Let $j\neq 0$ be the integer such that $\rho \in \mathcal V_j$. Let $g_0$ be the hyperbolic metric of same volume which is conformal to $g$. Using the normalized Ricci flow, it is easy to see that there exists a smooth family of negatively curved metrics 
  $g(t)$ such that $g(0)=g_0$ and $g(1)=g$. Indeed, the normalized Ricci flow starting from a negatively curved metric on a surface is known to exists in positive time, to preserve negative curvature, the volume and the conformal class of the metric. Finally, it is known to converge exponentially fast to the unique hyperbolic metric of same volume in the conformal class of $g$, see \cite[Theorem 3.3]{Ham88}. In particular, we can flow along the Ricci flow for a long time until we are very close to $g_0$, and since the space of negatively curved metrics is open, this shows the existence of the desired path.
  By \cite[Theorem 5]{CD24}, the dynamical torsions of $(g,\rho)$ and $(g_0,\rho)$ are equal. Since $\mathcal U_g\cap \mathcal V_j$ and $\mathcal U_{g_0}\cap \mathcal V_j$ both have complements of complex codimension $\geq 1$ in $\mathcal V_j$, we deduce that $\mathcal U_{g_0}\cap \mathcal U_g\cap \mathcal V_j\neq \emptyset.$ Let $\rho_0$ be a representation in the intersection. By connectedness of $\mathcal{V}_j$, see Proposition \ref{prop:connected2}, we can choose a smooth family of representations $\rho(t)\in \mathcal{V}_j$ such that $\rho(0)=\rho_0$ and $\rho(1)=\rho$. Note that this is a path of acyclic representations by Lemma \ref{lemm:coho}. By \cite[Theorem 6]{CD24}, we deduce that the dynamical torsions of $(g_0,\rho)$ and $(g_0,\rho_0)$ are equal. Since $\rho\in \mathcal U_g$ and $\rho_0\in \mathcal{U}_{g_0}$, Theorem \ref{theo2} implies that $0$ is not a resonance for $(g,\rho)$ or $(g_0,\rho_0)$ and the dynamical torsion of $(g,\rho)$ (resp. $(g_0,\rho_0)$) thus coincides with $\zeta_{g,\rho}(0)^{-1}$ (resp. $\zeta_{g_0,\rho_0}(0)^{-1}$) up to a sign. By \cite[Theorem A]{BFS23}, we conclude that
  \begin{equation*}
      \zeta_{g,\rho}(0)^{-1}=\pm \zeta_{g_0,\rho_0}(0)^{-1}=\pm \tau_{\mathfrak{e}_{\mathrm{geod}},\mathfrak{o}}(\rho)= \pm \det(\id-\rho(c))^{2G-2}.\qedhere
  \end{equation*}
\end{proof}

\section{Proofs of Theorem \ref{theo:3} and \ref{theo:4}}
\label{sec:last}
\subsection{Proof of Theorem \ref{theo:3}}
In this subsection, we prove Theorem \ref{theo:3}. The idea is to rewrite the dynamical determinant $\zeta_{0,\mathrm{Ad}}$ (see \eqref{eq:zeta_l}) in terms of the \emph{twisted Selberg zeta function}. The zeros of this zeta function were studied in \cite{NS22}.
\begin{proof}
Let $\rho \in \mathrm{Hom}(\pi_1(M),\mathrm{GL}_r(\CC)).$
For $s\in \mathbb C$ such that $\mathrm{Re}(s)\gg 1,$ the \emph{twisted Selberg zeta function} is equal to
    \begin{equation} \label{eq:Selberg}Z(s,\rho)=\prod_{k=0}^{+\infty}\prod_{\gamma\in \mathcal{P}} \det(\mathrm{Id}-\rho([\gamma])e^{-(s+k)\ell_g(\gamma)})=\exp\left(-\sum_{k=0}^{+\infty}\sum_{\gamma\in \mathcal{P}}\sum_{j=1}^{+\infty}\frac 1 j \mathrm{tr}\big(\rho([\gamma])^j \big)e^{-(s+k)j\ell_g(\gamma)}\right).
    \end{equation}
    Since $(\Sigma,g)$ is hyperbolic, one can compute explicitly the determinant of the Poincaré map:
    \begin{align*}\forall \gamma \in \mathcal P, \ \forall j\geq 1,\quad |\mathrm{det}(\mathrm{Id}-\mathcal P_{\gamma^j})|^{-1}&=\big(e^{j\ell_g(\gamma)}-1 \big)^{-1}\big(1-e^{-j\ell_g(\gamma)}\big)^{-1}=e^{-\ell_g(\gamma)}\frac{1}{\big(1-e^{-j\ell_g(\gamma)}\big)^2}
    \\
    &=e^{-\ell_g(\gamma)}\sum_{k=0}^{+\infty}(k+1)e^{-k\ell_g(\gamma)}=\sum_{k=0}^{+\infty}(k+1)e^{-(k+1)\ell_g(\gamma)}.
    \end{align*}
    In particular, we obtain from \eqref{eq:zeta_l} that
    \begin{align*}
        \zeta_{0,\rho}(s)&=\exp\left(-\sum_{\gamma\in \mathcal{P}}\sum_{j=1}^{+\infty}\frac 1 j \mathrm{tr}\big(\rho([\gamma])^j \big)\sum_{k=0}^{+\infty}(k+1)e^{-(k+1+s)\ell_g(\gamma)}\right).
    \end{align*}
    Exchanging the order of summation and using \eqref{eq:Selberg} gives for $s\in \CC$ such that $\mathrm{Re}(s)\gg 1,$
$$\zeta_{0,\rho}(s)=\prod_{k=1}^{+\infty}Z(s+k,\rho). $$
The relation extends to $s\in \CC$ using the meromorphic extensions of $\zeta_0(s,\rho)$ and $Z(s,\rho).$
The Selberg zeta function for $\rho=\mathrm{Ad}$ was studied in \cite{NS22}. In particular, one has
$$Z(s,{\mathrm{Ad}})=Z(s-1)Z(s)Z(s+1), $$
where $Z(s)$ is the (untwisted) Selberg zeta function of $(\Sigma,g),$ see \cite[\S 3.1]{NS22}. In total, we have obtained
\begin{equation}
    \label{eq:adj}
    \zeta(s, \mathrm{Ad})=Z(s)Z(s+1)^2\prod_{k=2}^{+\infty}Z(s+k)^3.
\end{equation}
Now, $Z(s)$ does not vanish in $\{s \in \mathbb C \mid \mathrm{Re}(s)>1\}$ and vanishes at $s=0$ to order $1+(2G-2)$ and at $s=1$ to order $1$. Since the algebraic multiplicity of $0$ as a Pollicott--Ruelle resonance is equal to the multiplicity of $0$ as a zero of $\zeta_0(s,\mathrm{Ad})$, we conclude that
$$\mathrm{dim}(\Res_{0}^{0,\infty}(\mathrm{Ad}))=1+(2G-2)+2=2G+1. $$
We can use \eqref{eq:orderofvan} and \cite[Corollary C]{FS23} to obtain
$$ \mathrm{dim}(\Res_{0}^{1,\infty}({\rm Ad},0))=3(2G-2)+2(2G+1)=10G-4.$$
This concludes the proof of Theorem \ref{theo:3}.
\end{proof}
\subsection{Proof of Theorem \ref{theo:4}}
\label{sec:JD}
In this subsection, we show Theorem \ref{theo:4}. 
\begin{proof} Let $(\Sigma,g)$ be an orientable connected closed hyperbolic surface. Let $\pi:\mathrm{SL}_2(\mathbb{R})\to \mathrm{PSL}_2(\mathbb{R})$ be the projection. Let us consider
\begin{equation*}
\Gamma:=\pi_1(\Sigma) \subset \mathrm{PSL}_2(\RR),\quad \Tilde{\Gamma}=\pi^{-1}(\Gamma) \subset \mathrm{SL}_2(\mathbb{R}).
\end{equation*}

Note that $\Tilde{\Gamma}$ acts by left multiplication on $\mathbb{C}^2$, so we can define
\begin{equation*}
\Tilde{\mathcal{E}}:=\mathrm{SL}_2(\mathbb{R})\times \CC^2/\sim,\quad ({x},v)\sim (\tilde \gamma \cdot  x,\tilde \gamma v), \ \forall \tilde \gamma\in \tilde \Gamma, \ \forall  x\in \mathrm{SL}_2(\RR), \ \forall v\in \CC^2.
\end{equation*}
This is a rank-$2$ vector bundle over $\tilde{\Gamma}\backslash \mathrm{SL}_2(\RR)\cong \Gamma \backslash \mathrm{PSL}_2(\RR)=M$. The trivial connection on $\mathrm{SL}_2(\mathbb{R})\times \CC^2$ descends to a flat connection on $\Tilde{\mathcal{E}}$. We can define the associated representation $\tau :\pi_1(M)\to \mathrm{GL}_2(\CC)$ by setting, for any $[\gamma]\in \pi_1(M)$, $\tau([\gamma])$ to be the parallel transport for the flat connection along $\gamma$, a representative of $[\gamma]$. Since the connection is flat, this does not depend on the choice of $\gamma$. Since $\tilde{\Gamma}$ is Zariski dense in $\mathrm{SL}_2(\RR)$ and the left multiplication of $\mathrm{SL}_2(\RR)$ on $\CC^2$ is irreducible, we know $\tau$ is irreducible.
\begin{lem}
    The representation $\tau$ does not factor through $\pi_1(\Sigma).$
\end{lem}
\begin{proof}
    Let $(x,v)\in M$ and define $s_{x,v}:[0,2\pi]\to S_x\Sigma$ such that $s_{x,v}(\theta):=(x,R_\theta v)$ where $R_\theta=\cos(\theta)+\sin(\theta)J$ is the rotation in the circle fiber defined by the complex structure $J.$ The curve $s_{x,v}$ is closed, generates $\pi_1(\mathbb S^1)\cong \mathbb Z$ and $\pi_*s_{x,v}=0$. Hence, $\tau$ factors through $\pi_1(\Sigma)$ if and only if $\tau([s_{x,v}])=\mathrm{Id}$, see \S \ref{sec:pi1}. The rotation $R_\theta$ in the identification $S\mathbb H^2\cong \mathrm{PSL}_2(\RR)$ is given by the multiplication by $$R(\theta):=\begin{pmatrix}
        \cos(\theta/2)&\sin(\theta/2)\\
        -\sin(\theta/2)&\cos(\theta/2)
        \end{pmatrix}.$$
        Note that $R(2\pi)=-\mathrm{Id}$ is trivial in $\mathrm{PSL}_2(\RR)$ but \emph{not} in $\mathrm{SL}_2(\RR).$ In particular, the lift $\tilde{s}_{x,v}$ of $s_{x,v}$ to $\mathrm{SL}_2(\RR)$ connects $\mathrm{Id}$ to $-\mathrm{Id}$. Since the $\tilde{\Gamma}$-action defining $\tilde{\mathcal E}$ is the left multiplication, we deduce that $\tau([s_{x,v}])$ acts as $-\mathrm{Id}$ on $\CC^2$, which shows that $\tau$ does not factor through $\pi_1(\Sigma).$
\end{proof}

 A section $s\in C^{\infty}( M;\Tilde{\mathcal{E}})$ identifies with  a function $s\in C^{\infty}(\mathrm{SL}_2(\RR), \CC^2)$ which is $\tilde{\Gamma}$-equivariant, that is
    $$\forall x\in \mathrm{SL}_2(\RR),\ \forall \tilde \gamma \in \tilde{\Gamma},\quad f(\tilde \gamma\cdot x)=\tilde \gamma \cdot f(x). $$
 We start by showing that $\Tilde{\mathcal{E}}$ splits as the direct sum of two invariant trivial line bundles $\Tilde{\mathcal{E}}^\pm$.  We define two sections $s_\pm \in C^{\infty}( M;\Tilde{\mathcal{E}})$ by defining $s_+(x)$ (resp. $s_-(x)$) to be the first column (resp. second column) of $x\in \mathrm{SL}_2(\RR)$. We check the equivariance property:
    $$\forall \tilde \gamma \in \tilde{\Gamma}, \ \forall x \in \mathrm{SL}_2(\RR), \quad  s_\pm(\tilde \gamma \cdot x)=\tilde \gamma \cdot s_\pm(x),$$
    by the basic rules of matrix multiplication.
     Recall that in the identification $S\mathbb H^2\cong \mathrm{PSL}_2(\RR)$, the action of the geodesic flow is given by 
    $$\forall t\in \RR, \ \forall [ x]\in \mathrm{PSL}_2(\RR)\quad \varphi_t^g([x])=[x] \cdot \begin{pmatrix}
        e^{t/2}&0\\
        0&e^{-t/2}
        \end{pmatrix}, $$
         which lifts to $\mathrm{SL}_2(\RR)$ to the right multiplication by $\mathrm{diag}(e^{t/2},e^{-t/2})$.
        Define $\Tilde{\mathcal{E}}^\pm:=\mathrm{span}(s_\pm)$. Then $\Tilde{\mathcal{E}}^\pm$ are trivial bundles which are invariant by the geodesic flow:
        \begin{equation}
            \label{eq:invsec}
            \forall t\in \mathbb R, \ \forall x\in \mathrm{SL}_2(\RR), \quad s_\pm(\varphi_t^g(x))=e^{\pm t/2}s_\pm(x).
        \end{equation} 
        See $f\in C^{\infty}( M;\Tilde{\mathcal{E}})$  as an equivariant function in $C^{\infty}(\mathrm{SL}_2(\mathbb R), \CC^2)$. Let $f_\pm \in C^{\infty}(\mathrm{SL}_2(\mathbb R),\CC)$ be two functions such  that
\begin{equation*}
\forall x\in \mathrm{SL}_2(\mathbb R), \ \forall \tilde \gamma \in \tilde{\Gamma},\quad     f(x)=f_+(x) s_+(x)+f_-(x)s_-(x),\quad f_{\pm}(\tilde \gamma (x))=f_{\pm}(x).
\end{equation*}
Note that the invariance condition on $f_\pm$ shows that they actually define functions on  $M$.
Then for any $t\in \mathbb R$, one has
\begin{equation}
\label{eq:11}
\begin{split}
    f(\varphi_t^g(x))&=f_+(\varphi_t^g(x)) s_+(\varphi_t^g(x))+f_-(\varphi_t^g(x))s_-(\varphi_t^g(x))
    \\&=e^{t/2}f_+(\varphi_t^g(x)) s_+(x)+e^{-t/2}f_-(\varphi_t^g(x))s_-(x).
    \end{split}
\end{equation}
Differentiating \eqref{eq:11} at $t=0$, we find
\begin{equation}
\label{eq:important}
    \mathcal L_{X^\tau}(f)(x)=\frac{d}{dt} f(\varphi_t^g(x))|_{t=0}=(X+1/2)f_+(x) s_+(x)+(X-1/2)f_-(x) s_-(x).
\end{equation}
In other words, the action of $\mathcal L_{X^\tau}$ on $0$-forms is the same as the (non-twisted) action of $X$ on functions (on $M$) shifted by $\pm \tfrac 12$. Therefore, to show that there is a Jordan block for the action of $\mathcal L_{X^{\tau}}$ at zero on $0$-forms, it suffices to show that $X$ has a Jordan block on $0$-forms at $-\tfrac 12$. We will use  \cite{GHW} which describes the Jordan block structure of the Pollicott--Ruelle resonances on $0$-forms. The main result of this section is the following proposition.
\begin{prop}
    \label{prop:JDtau}
 Let $(\Sigma,g)$ be an orientable connected closed hyperbolic surface and let $\tau$ be the representation of $\pi_1(M)$ constructed above. For $k=0,1,2$, the Jordan block at zero is at most of order $2$, that is, $\mathrm{Res}^{k,\infty}_0(\tau,0)=\mathrm{Res}^{k,2}_0(\tau,0)$. Moreover, we have
 \begin{equation}
     \label{eq:fulldescriptionJD}
     \begin{split}
   &\mathrm{dim}(\mathrm{Res}^{k,1}_0(\tau,0))=\mathrm{dim}(\mathrm{ker}(\Delta_g-\tfrac 14))    ,\quad  \mathrm{dim}(\mathrm{Res}^{k,\infty}_0(\tau,0))=2\mathrm{dim}(\mathrm{ker}(\Delta_g-\tfrac 14)), \ k=0,2,
   \\
   &\mathrm{dim}(\mathrm{Res}^{1,1}_0(\tau,0))=2\mathrm{dim}(\mathrm{ker}(\Delta_g-\tfrac 14))    ,\quad  \mathrm{dim}(\mathrm{Res}^{1,\infty}_0(\tau,0))=4\mathrm{dim}(\mathrm{ker}(\Delta_g-\tfrac 14)).
     \end{split}
 \end{equation}
 In particular, one has $m(g,\tau)=0$ and there is a non trivial Jordan block at zero for $k=0,1,2$ if and only if $\tfrac{1}{4}\in \mathrm{Spec}(\Delta_g).$
\end{prop}

\begin{proof}
   Let us start with $k=0$. From \eqref{eq:important}, it is clear that $$\mathrm{Res}_0^{0,\bullet}(\tau,0)\cong \mathrm{Res}_0^{0,\bullet}(X,\tfrac 12)\oplus \mathrm{Res}_0^{0,\bullet}(X,-\tfrac 12),\quad \bullet\in \mathbb N \cup \{+\infty\}.$$
   The Pollicott--Ruelle spectrum of $X$ lies in the half plane $\{\mathrm{Re}(\lambda)\leq 0\}$ so the first term is equal to $\{0\}$. For the second term, we use \cite[Theorem 3.3, 2)]{GHW} which gives that the Jordan block at $-1/2$ is at most of order $2,$ that is, $\mathrm{Res}_0^{0,\infty}(X,-\tfrac 12)=\mathrm{Res}_0^{0,2}(X,-\tfrac 12)$. Moreover, if we denote by $U_-\in C^{\infty}(M;TM)$ the unstable horocycle operator, we have
   \begin{align*}
        &\mathrm{dim}(\mathrm{Res}^{0,1}_0(X,-\tfrac 12)\cap \ker(U_-))=\mathrm{dim}(\mathrm{ker}(\Delta_g-\tfrac 14))    ,
        \\&\mathrm{dim}(\mathrm{Res}^{0,\infty}_0(X,-\tfrac 12)\cap \ker(U_-))=2\mathrm{dim}(\mathrm{ker}(\Delta_g-\tfrac 14)).
   \end{align*}
Note that  since $|-\tfrac 12|<1$, \cite[Proposition 1.3]{GHW} implies that $\mathrm{Res}^{0,\bullet}_0(X,-\tfrac 12)=\mathrm{Res}^{0,\bullet}_0(X,-\tfrac 12)\cap \ker(U_-)$ for $\bullet =1,\infty$. In other words, all (generalized) resonant states at $-1/2$ are in the first band. This readily implies \eqref{eq:fulldescriptionJD} for $k=0$. Since $u\mapsto u\wedge d\alpha$ defines an isomorphism  $\mathrm{Res}^{0,\infty}_0(\tau,0)\to \mathrm{Res}^{2,\infty}_0(\tau,0)$ that preserves the Jordan block structure, we also deduce \eqref{eq:fulldescriptionJD} for $k=2$.

Suppose now that $k=1$. Since $g$ is hyperbolic, $\Omega^1_0$ is trivialized by two smooth sections $\beta_\pm$ which satisfy $\mathcal{L}_X\beta_\pm =\pm \beta_\pm$. These $1$-forms are just the dual forms to the stable and unstable horocyclic vector fields. As a consequence, we deduce from \eqref{eq:important} that for any $\quad \bullet\in \mathbb N \cup \{+\infty\},$ 
\begin{align*}
    \mathrm{Res}_0^{1,\bullet}(\tau,0)&\cong \mathrm{Res}_0^{0,\bullet}(\tau,1)\oplus \mathrm{Res}_0^{0,\bullet}(\tau,-1)
    \\&\cong\mathrm{Res}^{0,\bullet}_0(X,\tfrac 32)\oplus \mathrm{Res}^{0,\bullet}_0(X,\tfrac 12)\oplus \mathrm{Res}^{0,\bullet}_0(X,-\tfrac 12)\oplus \mathrm{Res}^{0,\bullet}_0(X,-\tfrac 32)
    \\&=\mathrm{Res}^{0,\bullet}_0(X,-\tfrac 12)\oplus \mathrm{Res}^{0,\bullet}_0(X,-\tfrac 32),
 \end{align*}   
where we used again that the Pollicott--Ruelle resonances have non-positive real part.  From the first part of the proof, we know that the term $\mathrm{Res}^{0,1}_0(X,-\tfrac 12)$ contributes to a $\mathrm{dim}(\mathrm{ker}(\Delta_g-\tfrac 14))$-dimensional space of resonant states and to a $2\mathrm{dim}(\mathrm{ker}(\Delta_g-\tfrac 14))$-dimensional space of generalized resonant states. Moreover, the Jordan block is of order at most $2.$ We now investigate the contribution of $\mathrm{Res}^{0,1}_0(X,-\tfrac 32).$ From \cite[Proposition 1.3]{GHW}, we have
$$\mathrm{Res}_0^{0,\infty}(X,-\tfrac 32)=(\mathrm{Res}_0^{0,\infty}(X,-\tfrac 32)\cap \mathrm{ker}(U_-))\oplus U_+\left(\mathrm{Res}_0^{0,\infty}(X,-\tfrac 12)\cap  \mathrm{ker}(U_-)\right), $$
where $U_+\in C^{\infty}(M;TM)$ denotes the stable horocycle operator.
Note that the map $U_+:\mathrm{Res}_0^{0,\infty}(X,-\tfrac 12)\cap  \mathrm{ker}(U_-)\to \mathrm{Res}_0^{0,\infty}(X,-\tfrac 32)$ is injective by \cite[Proposition 1.3]{GHW} since $-\tfrac 32 \notin \mathbb Z$. Moreover, it preserves the Jordan block structure. The second term hence contributes to a $\mathrm{dim}(\mathrm{ker}(\Delta_g-\tfrac 14))$-dimensional space of resonant states and to a $2\mathrm{dim}(\mathrm{ker}(\Delta_g-\tfrac 14))$-dimensional space of generalized resonant states. To conclude the proof of the proposition, we only need to show that $\mathrm{Res}_0^{0,\infty}(X,-\tfrac 32)\cap \mathrm{ker}(U_-)=\{0\}.$ We use \cite[Theorem 3.3, 1)]{GHW} which implies that $\mathrm{Res}_0^{0,\infty}(X,-\tfrac 32)\cap \mathrm{ker}(U_-)$ is isomorphic to $\mathrm{ker}(\Delta_g+(-\tfrac 32)(1-\tfrac{3}{2}))=\mathrm{ker}(\Delta_g+\tfrac 34)=0$ since $\Delta_g$ is a non-negative operator.
\end{proof}
We show that there exists examples of hyperbolic surfaces for which $1/4$ is in the spectrum of the Laplace operator of any genus $G\geq 2$. 
\begin{prop}
    \label{prop:1/4}
    Let $G\geq 2$, then there exists a hyperbolic surface $(\Sigma,g)$ of genus $G$ for which $\tfrac 14 \in \mathrm{Spec}(\Delta_g)$.
\end{prop}
\begin{proof}
 Let $\Sigma$ be a closed surface of genus $G\geq 2.$ For a hyperbolic metric $g$ on $\Sigma,$ let $\lambda_1(g)$ denote the first positive eigenvalue of $\Delta_g$. It is known that $g\mapsto \lambda_1(g)$ is continuous when $g$ varies in the Teichm\"uller space, see for instance \cite{BU83}. The Bolza surface is a genus two hyperbolic surface for which $\lambda_1\cong 3.838887$, see \cite{SU17}. On the other hand, Buser's inequality implies that $\lambda_1(g)\leq C(h(g)+h(g)^2)$, where $h(g)$ is the \emph{Cheeger constant} of $g$ and $C>0$ is a universal constant. Since there exists hyperbolic metrics on $\Sigma$ with arbitrarily small Cheeger constant, we deduce that $\lambda_1(g)$ can be made arbitrarily small in the Teichm\"uller space. Therefore, there exists a hyperbolic manifold $(\Sigma,g_0)$ of genus $2$ for which $\lambda_1(g_0)=\tfrac 14$ since the Teichm\"uller space is connected.

 Now, let $\pi:(N,h)\to (M,g_0)$ be a connected cover of $(M,g_0)$ of degree $d\geq 1$. The existence of such a cover is guaranteed since $\pi_1(\Sigma)$ has $F_2$ (the free group with $2$ generators) as a quotient, which has a representation to $S_d$ which is transitive on $\{1,2,\cdots,d\}$.  The Euler characteristic of $N$ is 
 $$ \chi(N)=d\chi(M)=d(2-2\times 2)=-2d.$$
 In particular, the genus $G_N$ of $N$ is equal to
 $G_N=1-\tfrac 12 \chi(N)=d+1. $ Since $\tfrac 14\in \mathrm{Spec}(\Delta_{g_0})$, we deduce that $\tfrac 14\in \mathrm{Spec}(\Delta_h)$ and this provides an example in any genus $G\geq 2.$
\end{proof}
\end{proof}

\section{Generic semisimplicity}
\label{sec:GenSim}
In this section, we prove Theorem~\ref{thm:sim} following a similar strategy as in the proof of Theorem~\ref{theo2}. Let $(\Sigma,g)$ be an orientable connected closed Anosov manifold of dimension $n+1\geq 2.$ For a finite dimensional representation $\rho$ of $\pi_1(M)$, let
\begin{equation*}
        d_k := d^{\nabla_{\rho}}: \Res_0^{k,1}(g,\rho,0) \to \Res_0^{k+1,1}(g,\rho, 0).
    \end{equation*}
The map $d_k$ is well defined by a similar argument as in the proof of Proposition \ref{prop:noJD}. Following \cite[Equation (2.61)]{zeta3}, we consider
\begin{equation}\label{eq:def-pik}
    \pi_k: \Res_0^{k,1}(g,\rho, 0)\cap \ker d_k \to H^k(M,\rho).
\end{equation}
The proof of \cite[Lemma 2.6]{zeta3} extends to our setting and we have
\begin{equation}\label{eq:cddp6}
    d^{\nabla_{\rho}}(\Res_0^{k-1,1}(g,\rho, 0))\subset \ker \pi_k \subset d^{\nabla_{\rho}}(\Res^{k-1,\infty}(g,\rho, 0)).
\end{equation}
In the previous chain of inclusion, we have denoted by 
\begin{equation}
\label{eq:genresno0}
\Res^{k,\infty}(g, \rho,\lambda_0)=\{u\in \mathscr{D}'(M;\mathscr E^k_{\rho}): \exists \ell\in \mathbb{N},\ (\mathcal{L}_{X^\rho}-\lambda_0)^{\ell} u = 0 , \, \WF(u)\subset E_u^* \},
\end{equation}
the set of generalized resonant states which are not necessarily in the kernel of $\iota_X.$
Moreover, the proof of \cite[Lemma 2.7]{zeta3} shows that $\pi_k$ is onto if $\Res_{0}^{*,2n+1-k,1}(g,\rho,0)$ consists of exact forms.

Let us introduce the dual conditions to \eqref{eq:cond-simp} for coresonant states:
\begin{equation}\label{eq:cond-cosimple}
     \Res_0^{*,k,\infty}(g,\rho,0)=\Res_0^{*,k,1}(g,\rho,0),\quad
    d^{\nabla_{\rho^*}}(\Res_0^{*,k,1}(g,\rho,0))=0,\quad k=0,1,2,\cdots, 2n.
\end{equation}
Below is our key lemma for induction.
\begin{lem}\label{lem:ind}
Let $0\leq k \leq n$. Suppose both \eqref{eq:cond-simp} and \eqref{eq:cond-cosimple} hold for $j$-forms for $j=0,1,\cdots, k-1$, then we have
    \begin{equation}\label{eq:lower}
        \dim ( \Res_0^{k,1}(g,\rho,0) \cap \ker d_k ) = \sum_{j=0}^{\lfloor k/2\rfloor} \dim H^{k-2j}(M,\rho) 
    \end{equation}
and
\begin{equation}\label{eq:colower}
        \dim ( \Res_0^{*,k,1}(g,\rho,0) \cap \ker d_k ) = \sum_{j=0}^{\lfloor k/2\rfloor} \dim H^{k-2j}(M,\rho^*) .
    \end{equation}
   
\end{lem}
\begin{proof}
We only prove \eqref{eq:lower} as the argument to show \eqref{eq:colower} is similar.  Consider $\pi_k$ defined in \eqref{eq:def-pik}. By \eqref{eq:cddp6}, and the fact that there is no Jordan block on $(k-1)$-forms, we deduce
\begin{equation*}
    \ker \pi_k \subset d^{\nabla_{\rho}}(\Res^{k-1,1}(g,\rho,0))= \Res_0^{k-2,1}(g,\rho,0) \wedge d\alpha.
\end{equation*}
For the equality, we used that $\Res^{k-1,1}(g,\rho,0)=\Res^{k-1,1}_0(g,\rho,0)\oplus \Res^{k-2,1}(g,\rho,0)\wedge \alpha $ and the fact that resonant states on $(k-2)$-forms and on $(k-1)$-forms are closed by hypothesis. Then it follows that
\begin{equation*}
    \ker \pi_k =\Res_0^{k-2,1}(g,\rho,0) \wedge d\alpha,
\end{equation*}
since the other inclusion is obvious.
On the other hand, one has (see \cite[(6-2)]{CD24})
\begin{equation}
\label{eq:exact}
    \Res_{0}^{*,2n+1-k,1}(g,\rho,0)= \Res_{0}^{*,k-1,1}(g,\rho, 0)\wedge (d\alpha)^{n+1-k}
\end{equation}
is exact since the coresonant states on $(k-1)$-forms are closed by assumption. So $\pi_k$ is onto by \cite[Lemma 2.7]{zeta3}. In particular, applying the rank-nullity theorem to $\pi_k$ gives
$$\dim ( \Res_0^{k,1}(g,\rho,0) \cap \ker d_k )=\dim ( \Res_0^{k-2,1}(g,\rho,0) )+\dim H^{k}(M,\rho). $$
We deduce \eqref{eq:lower} using an induction. 
\end{proof}

We now show Theorem~\ref{thm:sim}. 
\begin{proof}
Let $\mathcal V$ be an open and connected set of Anosov metrics.
We define 
\begin{equation}
    \mathcal{U}=\{g \in \mathcal{V} : g \text{ satisfies } \eqref{eq:cond-simp} \text{ and } \eqref{eq:cond-cosimple}\}.
\end{equation}
Our goal is to show that $\mathcal U\subset \mathcal V$ is open and dense whenever it is not empty.

First we show $\mathcal{U}$ open. For this, let $g_0\in \mathcal{U}$. By the perturbation theory of Ruelle resonances, see \cite{Bon}, the dimension of the space of generalized resonant states can only decrease for a perturbation. In other words, in a neighborhood of $g_0$ we have
\begin{equation}\label{eq:ind-leq}
    \dim \Res_0^{k,\infty} (g,\rho,0 )\leq \dim \Res_0^{k,1}(g_0,\rho,0),\quad k=0,1,\ldots,2n.
\end{equation}
Using Lemma~\ref{lem:ind}, we now show that $g$ must satisfy \eqref{eq:cond-simp} and \eqref{eq:cond-cosimple}. Indeed, for $k=0$, the assumption of Lemma~\ref{lem:ind} is empty, and we have
$$\dim ( \Res_0^{0,\infty}(g,\rho,0)) \geq \dim ( \Res_0^{0,1}(g,\rho,0) \cap \ker d_0 )= \dim H^{0}(M,\rho), $$
and thus by \eqref{eq:ind-leq} we have $\dim ( \Res_0^{0,\infty}(g,\rho,0))=\dim ( \Res_0^{0,1}(g,\rho,0) \cap \ker d_0 )= \dim H^{0}(M,\rho)$. This shows that \eqref{eq:cond-simp} holds for $k=0$. Similarly, \eqref{eq:cond-cosimple} holds for $k=0$. Next, Lemma~\ref{lem:ind} shows that $$
\dim ( \Res_0^{1,\infty}(g,\rho,0))\geq \dim ( \Res_0^{1,1}(g,\rho,0) \cap \ker d_1 )= \dim H^{1}(M,\rho).$$
Thus, by \eqref{eq:ind-leq} we know \eqref{eq:cond-simp} holds for $k=1$. Inductively, this gives that \eqref{eq:cond-simp} and \eqref{eq:cond-cosimple} hold for all $k=0,\cdots, 2n$. This proves that $\mathcal{U}$ is open.

Now suppose $\mathcal{U}$ is not empty, say $g_0\in \mathcal{U}$. We prove $\mathcal{U}$ is dense in $\mathcal{V}$. Let $g\in \mathcal{V}$, we choose a piecewise linear family of metrics $g(t)$ such that $g(0)=g_0$ and $g(1)=g$. To construct such a path, first consider a path $\tilde{g}(t)$ in $\mathcal V$ such that $\tilde{g}(0)=g_0$ and $\tilde{g}(1)=g_1$. Since $\mathcal V$ is open and $[0,1]$ is compact, there exists $\epsilon>0$ such that $B(\tilde g(t),\epsilon)\subset \mathcal V$ for all $t\in [0,1]$. Then we take $0=t_0<t_1<\cdots< t_N=1$ such that the distance between $\tilde{g}(t_i)$ and $\tilde{g}(t_{i+1})$ is less than $\epsilon/2$ for $i=0,1,\ldots, N-1$. In particular, the piecewise linear path joining $\tilde{g}(t_i)$ and $\tilde{g}(t_{i+1})$  for $i=0,\ldots, N-1$ stays in $\mathcal V.$

Consider the first linear path $(g_0, \tilde g(t_1))$ from $g_0$ to $\tilde g(t_1)$. Using Lemma \ref{lem:5.2} and the fact that the spectral projector varies real analytically on this segment, we deduce that for any $k \in [0,n]$, the subset 
$$\mathcal{A}_k:=\{ g\in (g_0, \tilde g(t_1))\mid \dim ( \Res_0^{k,1}(g,\rho,0)) \geq \sum_{j=0}^{\lfloor k/2\rfloor} \dim H^{k-2j}(M,\rho)+1 \} $$
is either equal to $(g_0, \tilde g(t_1))$ or finite. Since $g_0\in \mathcal U$ and $\mathcal U$ is open, this implies that $\mathcal A_k$ is finite for any $k \in [0,n]$. Applying the same argument to the dual representation then gives
\begin{equation*}
\{g\in (g_0, \tilde g(t_1))\mid g\notin \mathcal U\} \text{ is finite.}
\end{equation*}
Up to changing the endpoint $\tilde g(t_1)$ by a close point on $(g_0, \tilde g(t_1))$, we can suppose that $\tilde g(t_1)\in \mathcal U$. Applying the same argument to the second linear path $(\tilde g(t_1),\tilde g(t_2))$ shows that 
$$\{g\in (\tilde g(t_1), \tilde g(t_2))\mid g\notin \mathcal U\} \text{ is finite.} $$
This shows that there are at most finitely many points on the piecewise linear path which are not in $\mathcal U$, and hence $\mathcal{U}$ is dense in $\mathcal{V}$.

Since \eqref{eq:cond-simp} implies \eqref{eq:cond-cosimple} by duality, this concludes the proof of Theorem~\ref{thm:sim}.
\end{proof}

\subsection{Perturbation of trivial representation}
In this section we discuss the perturbation of the trivial representation and prove Proposition~\ref{prop:per-triv}.

\begin{proof}[Proof of Proposition~\ref{prop:per-triv}]
By the assumption that $\dim H_1(M,\CC)\geq 2$,
we can take two elements $a_1,a_2\in H_1(M,\mathbb{Z})$ linearly independent in $H_1(M,\mathbb{Z})\otimes \CC =H_1(M,\CC)$ and two linear forms $f_1,f_2:H_1(M,\CC)\to \CC$ such that $f_i(a_j)=\delta_{ij}$ for $i,j=1,2$.

We consider the following holomorphic family of representations 
\begin{equation*}
    \rho(z_1,z_2): \pi_1(M)\to  H_1(M,\CC) \to \mathrm{GL}_r(\CC),\quad z_1,z_2\in \CC,
\end{equation*}
where the first map is sending a homotopy class to its homology class and the second map is given by $e^{z_1f_1}e^{z_2f_2} \id_r$. Then $\rho(z_1,z_2)$ is a holomorphic family of representations for $(z_1,z_2)\in \CC^2$. Note that for $z_1=z_2=0$, one has $\rho(0,0)=\rho_{\mathrm{triv}}\otimes \id_r$. Consider the determinant of the matrix $\widetilde{\Pi}^0_{\rho(z_1,z_2)}(0)\widetilde{\mathcal{L}}_{X^{\rho(z_1,z_2)}}\widetilde{\Pi}^0_{\rho(z_1,z_2)}(0)$. It is a holomorphic function of two variables which vanishes at $(z_1,z_2)=(0,0)$. By Hartogs theorem, its zero set cannot have isolated points. Hence there exists $(z_1,z_2)\neq(0,0)$ in a neighbourhood of $(0,0)$ such that
\begin{equation*}
    \Res_0^{0,\infty}(g,\rho(z_1,z_2),0)\neq 0.
\end{equation*}    
On the other hand, $H^0(M,\rho(z_1,z_2))=0$ for $(z_1,z_2)\neq (0,0)$ in the neighborhood of $(0,0)$ because (see Lemma~\ref{lemm:coho})
\begin{equation*}
     H^0(M,\rho)=H^0(M,\mathcal E_{ \rho}) =\{v\in \CC^r \mid \forall \gamma \in \pi_1(M),\ \rho(\gamma)v=v\}.
\end{equation*}
Suppose that $v\in \CC^r$ satisfies $\rho(\gamma)v=v$ for any $ \gamma \in \pi_1(M)$. Since the map $\pi_1(M)\to H_1(M,\ZZ)$ is surjective, we can consider $\gamma_i\in \pi_1(M)$ to be a pre-image of $a_i$ for $i=1,2$. Then we have 
$$\rho(z_1,z_2)(\gamma_i)v=e^{z_i}v=v, \quad i=1,2, $$
and since $(z_1,z_2)$ is close to $(0,0)$, this implies $(z_1,z_2)=(0,0)$ whenever $v\neq 0.$ This shows that $H^0(M,\rho(z_1,z_2))=0$ and concludes the proof.
\end{proof}

\def\arXiv#1{\href{http://arxiv.org/abs/#1}{arXiv:#1}}

\end{document}